\documentclass[onecolumn]{autart}   
\usepackage{fix-cm} 
\usepackage[fontsize=12pt]{fontsize}

\usepackage{xcolor}

 \usepackage{amsfonts}
\usepackage{natbib}        
\usepackage[skip=10pt plus1pt]{parskip}
\usepackage{amsmath}
\usepackage{amssymb}

 \usepackage{hyperref}
\usepackage{mathtools}
\usepackage{enumerate}

\usepackage{verbatim}
 \usepackage{setspace}

 \usepackage{subcaption}
 \usepackage{float}
 
 \newcommand{\Int}{\operatorname{int}}
 
 \newcommand{\real}{\operatorname{Re}}

\newcommand{\suchthat}{\, | \,}

\newcommand{\spanop}{\operatorname{span}}

\newcommand{\Poincare}{Poincar\'e}

 \newtheorem{proof}{Proof}
\newtheorem{Theorem}{Theorem}
 \newtheorem{Lemma}{Lemma}
\newtheorem{Corollary}{Corollary}
\newtheorem{Definition}{Definition}
\newtheorem{Proposition}[Theorem]{Proposition}
 \newtheorem{Remark} {Remark}
 \newtheorem{Example} {Example}

\newcommand{\beq}       {\begin{equation}}
\newcommand{\eeq}       {\end{equation}}
\newcommand{\bery}      {\begin{array}}
\newcommand{\eery}      {\end{array}}
\newcommand{\berys}     {\begin{array*}}
\newcommand{\eerys}     {\end{array*}}
\newcommand{\beqry}     {\begin{eqnarray}}
\newcommand{\eeqry}     {\end{eqnarray}}
\newcommand{\beqrys}    {\begin{eqnarray*}}
\newcommand{\eeqrys}    {\end{eqnarray*}}

\def \R {{\mathbb R}}
\def \B {{\mathcal B}}

\newcommand{\be}{\begin{equation}}
\newcommand{\ee}{\end{equation}}

\begin{document}
\begin{frontmatter}

\title{\LARGE   Instability of equilibrium  and  convergence to   periodic orbits   in    strongly 2-cooperative~systems   }

\thanks[footnoteinfo]{R.K. and G.G. acknowledge support by the European Union through the ERC INSPIRE grant (project n. 101076926).
Views and opinions expressed are however those of the authors only and do not necessarily reflect 
those of the~EU, the European Research Executive Agency or the European Research Council. Neither the~EU nor the granting authority can be held responsible for them. The research of M.M. is partially supported by the ISF.} 

\author[TN]{Rami Katz} 
\author[TN]{Giulia Giordano}
\author[TAU]{Michael Margaliot}

\address[TN]{Department of Industrial Engineering,
University of Trento, Italy.}
\address[TAU]{School of Electrical Engineering, Tel Aviv University, Israel.}

\begin{abstract}
We consider time-invariant 
nonlinear $n$-dimensional   strongly  $2$-cooperative systems, that is, systems that map the  set of vectors with up to  weak   sign variation to its interior. 
Strongly  $2$-cooperative systems   enjoy a strong  \Poincare-Bendixson property:  bounded solutions that maintain a positive distance from the set of 
equilibria   converge to a periodic solution. 
  For strongly  $2$-cooperative systems whose trajectories evolve in a bounded and invariant set that contains a single unstable equilibrium, we provide a simple criterion for the existence of periodic trajectories. Moreover, we explicitly characterize a positive-measure set of initial conditions which yield solutions that asymptotically converge to a periodic trajectory. 
We demonstrate our theoretical results using two models from systems biology,  the $n$-dimensional Goodwin oscillator and a $4$-dimensional biomolecular oscillator with RNA-mediated regulation, and provide numerical simulations that verify the  theoretical results. 
\end{abstract}


  \end{frontmatter}

\section{Introduction}

Oscillations, as well as periodic behaviors, rhythms and regular patterns, are among the most widespread behavioral motifs in nature \citep{Goldbeter2012}, but they are still not entirely understood.
Biological clocks and oscillators are frequently encountered at various scales. 
Single-molecule clocks \citep{Johnson2017} include the segmentation clock \citep{Uriu2009,Uriu2010}, which rules pattern generation in vertebrate embryonic development.
The regulation of cell life cycle and metabolism hinges upon biomolecular oscillators \citep{Ferrell2011}.
An internal biological clock regulates the physiological processes of living beings on Earth, including plants, animals and humans, according to circadian rhythms \citep{Doyle2006}, i.e. almost daily (\emph{circa dies}) fluctuations that are almost exactly synchronized with the 24-hour rotation period of the planet; this biological clock regularly oscillates even when the organism faces constant darkness, or constant light, and is entrained by the alternating daylight and darkness. The circadian clock thus allows living organisms to effectively adapt their physiology to the periodic rhythms of sunlight and darkness, and regulates e.g. hormone secretion, body temperature and metabolic functions.
Jeffrey C. Hall, Michael Rosbash and Michael W. Young received the 2017 Nobel Prize in Physiology or Medicine ``for their discoveries of molecular mechanisms controlling the circadian rhythm'' in fruit flies.
Also in synthetic biology, building biomolecular computing systems requires oscillators with a timekeeping function; several design principles and architectures have been proposed so far \citep{novak2008,syn_osci_2020}, but constructing reliable biomolecular oscillators with tunable amplitude and phase is still an open challenge.

Mathematical  control theory~\citep{sontag_book} offers powerful methodologies to regulate dynamical systems, but the  synthesis  and  analysis
    of 
biological oscillators   that exhibit periodic behaviors is still a major challenge. This type of questions is significant not only in the life sciences, but also in numerous applications in engineering, including vibrational mechanics \citep{Blekhman2000} and power electronics \citep{Schubert2016}.


  In general, it is important  to provide sufficient conditions guaranteeing that a bounded solution of an $n$-dimensional nonlinear system converges to a periodic orbit, thereby allowing for oscillatory behavior of solutions. 
Determining whether a set of nonlinear ODEs admits a periodic orbit, and if so, what is the set of initial conditions yielding convergence to a periodic orbit is a highly non-trivial problem~\citep{frakas_periodic_book}. 

We   focus on   a specific class of   nonlinear  dynamical systems that are strongly $2$-cooperative. These systems are known to  satisfy  a  strong \Poincare-Bendixson property:  if the $\omega$-limit set of a bounded   solution contains no equilibrium points, then it is a periodic orbit~\citep{WEISS_k_posi}.

  Here, we consider an $n$-dimensional  strongly 2-cooperative system evolving in a convex, invariant and bounded state space $\Omega$, which admits a unique equilibrium~$e\in \Omega$, such that the Jacobian of the vector field computed at~$e$ admits at least two unstable eigenvalues. Our main result shows  that the system admits an explicit set of initial conditions of positive measure, such that every solution emanating from this set converges to a periodic orbit. 
Our result generalizes recent work on the particular case of 3-dimensional strongly 2-cooperative systems~\citep{rami_3dGoddwin_2024}. 
The generalization to the   $n$-dimensional case 
requires several additional tools from the theory of cones of rank~$k$~\citep{Sanchez2009ConesOR}, 
the spectral theory of totally positive matrices (see, e.g.~\citep{total_book,fulltppaper}),
and a generalization of the Perron-Frobenius theorem to cones of rank~$k$~\citep{fusco_oliva_1991}. 

 To demonstrate the usefulness of the  theoretical result, we consider two important  examples from systems biology: the well-known $n$-dimensional Goodwin   system~\citep{GOODWIN1965425}, and a four-dimensional  biomolecular oscillator with RNA-mediated regulation 
proposed by~\citep{BCFG2014}. We rigorously analyze the oscillatory behavior of both systems, characterizing a set of initial conditions from which the trajectories converge to a periodic orbit, and we provide numerical simulations that validate our theoretical results. 

The remainder of this paper  is organized as follows. 
The next section reviews known
definitions and results that are  used throughout the paper. 
Section~\ref{sec:main_res} presents our main result. The proof of our main result is given in Section~\ref{sec:proof}.
Section~\ref{sec:applic} describes two applications from systems biology. 
The concluding section elaborates on possible directions for future research.  

\section{Notations and Preliminary results}
 Given an integer~$n\geq 1$, let  $[n]:=\left\{1,\dots,n \right\}$.
The cardinality of a set~$S$ is denoted by~$\operatorname{card}(S)$.
For a set~$S \subseteq \mathbb{R}^n$, we denote by~$\operatorname{int}\left(S \right)$ and~$\operatorname{clos}\left(S \right)$ its interior and closure, respectively. For two sets~$S_1,S_2 \subseteq \mathbb{R}^n $ and  a scalar $c \in \mathbb{R}$, we denote by $S_1+S_2 = \left\{s_1+s_2 \suchthat  s_i\in  S_i,\ i=1,2 \right\}$ and $c S_1 = \left\{c s_1  \suchthat
s_1\in  S_1 \right\}$ the usual set addition and scalar-set multiplication.  
We denote vectors and matrices by lowercase and uppercase letters, respectively. For~$\varepsilon>0$ and~$y\in \mathbb{R}^n$, let~$B(y,\varepsilon)$ denotes the open Euclidean ball centered at $y$ with radius $\varepsilon$. The non-negative orthant in~$\R^n$ is  
$\mathbb{R}^n_{\geq 0}:=\left\{x\in\R^n \suchthat x_i\geq 0,\  i\in [n] \right\}$,
and the non-positive orthant is $\mathbb{R}^n_{\leq 0}:=-\mathbb{R}^n_{\geq 0}$. The transpose, spectrum and determinant of a matrix $A$ are denoted by $A^\top$, $\operatorname{spec}(A)$ and $\det(A)$, respectively. $I_n$ is the~$n\times n$ identity matrix.  Given matrices $A_j\in \mathbb{R}^{n_j\times n_j},\ j\in [m]$, we denote by $\operatorname{diag}(A_1,\dots, A_m)$ the block diagonal matrix with diagonal elements $  A_j  $. A square matrix $A$ is \emph{Hurwitz} if all its eigenvalues have a negative real part, \emph{unstable} if it admits an eigenvalue with a positive real part, and  \emph{Metzler}  if all its off-diagonal entries are non-negative.

A matrix~$\bar A$ is a {\sl sign matrix} if every entry of~$\bar A$ is either~$*$ (``don't care''), $\leq 0$,   $0$,  or~$\geq 0$. A time-varying matrix~$A(t)$ has the sign pattern~$\bar A$ if the following three  properties hold at all times~$t\geq 0$:
\begin{enumerate}
\item $a_{ij}(t)\leq 0$  for all indices~$i,j$ such that~$\bar a_{ij}$ is $\leq 0$,
\item $a_{ij}(t)=0$  for all indices~$i,j$ such that~$\bar a_{ij}$ is $0$,
\item $a_{ij}(t)\geq 0$  for all indices~$i,j$ such that~$\bar a_{ij}$ is $\geq 0$,
\end{enumerate}
with no restriction  imposed  on~$a_{ij}(t)$  when~$\bar a_{ij}$ is $*$.

\subsection{From cooperative to~$k$-cooperative systems}

The flow of cooperative systems~\citep{hlsmith} maps the non-negative orthant $\mathbb{R}^n_{\geq 0}$ 
to itself, and also the non-positive orthant $\mathbb{R}^n_{\leq 0}$ to itself. 
In other words, the flow maps the set of vectors whose entries exhibit no sign changes (zero sign variations) to itself. Cooperative  systems have many special asymptotic properties: Hirsch's 
  quasi-convergence theorem asserts that in a strongly  cooperative system  with bounded solutions  all initial conditions outside of a zero-measure set give rise to solutions that converge to an equilibrium.

Cooperative systems and cooperative control systems~\citep{mcs_angeli_2003}, which allow to study interconnections of cooperative systems,  have found numerous applications in various  fields, including social dynamics~\citep{altafini_survey}, dynamic neural networks~\citep{smith_neural}, chemistry and systems biology~\citep{Angeli2004,BLANCHINI_GIORDANO_SURVEY,Donnell2009120,mono_chem_2007,RFM_stability,RFM_model_compete_J,sontag_near_2007}.
However, since a strongly cooperative system cannot admit an \emph{attracting} periodic orbit~\citep{hlsmith}, cooperative systems theory is not suitable to support the study of systems that admit periodic solutions.

The theory of~$k$-cooperative systems 
provides a generalization of 
  cooperative systems
\citep{WEISS_k_posi}. The flow of a $k$-cooperative system maps the set  of vectors with up to~$k-1$   weak  sign variations to itself. In particular, $1$-cooperative systems are cooperative systems. 
Also, $(n-1)$-cooperative systems, where~$n$ is the dimension of the system, are, up to a coordinate transformation,  competitive systems~\citep{WEISS_k_posi}. 
The property of $k$-cooperativity   depends on the \emph{sign structure}  (also called \emph{sign pattern}) of the system Jacobian, and hence it can be assessed without knowing the exact values of various system parameters. Below, we recall the formal definitions of these concepts.

\subsection{Sign variations in a vector}\label{sec:p2_homog}
\begin{Definition}\label{ref:sigVari}
Given the vector $x\in \mathbb{R}^n$ with $\prod_{i=1}^nx_i \neq 0$, the \emph{number of sign variations} in $x$ is denoted by
\begin{equation*}
    \sigma(x) := \operatorname{card}\left(\left\{i   \in [n-1] \suchthat  x_ix_{i+1}<0 \right\} \right).
\end{equation*} 
\end{Definition}
For example,   $\sigma\left(\left[2\pi~-3.3~ \sqrt{2}\right]^\top\right)=2$.

Two useful  generalizations of~$\sigma(\cdot)$ to vectors that may contain zero entries are offered by the theory of totally positive matrices (see, e.g.,~\cite{total_book,gk_book,pinkus}). 
\begin{Definition}\label{def:siggn_var}
Given~$x\in\R^n\setminus\{0\}$, let~${x}^d$ denote the vector
obtained from~$x$ by deleting all its zero entries. The \emph{weak number of sign variations} in~$x$ is 
\begin{equation*}
s^-(x):=\sigma({x}^d).
\end{equation*}
Let  $\mathcal{S}_x\subseteq \mathbb{R}^n$ denote  the set of all vectors obtained from~$x$ by replacing each of its zero entries by either~$+1$ or~$-1$. 
The  \emph{strong number of sign variations} in~$x$ is 
\begin{equation*}
s^+(x):=\max_{z\in\mathcal{S}_x} \sigma(z).   
\end{equation*} 
Finally, for $0\in \R^n$, we set~$s^-(0):=0$ and~$s^+(0):=n-1$.
\end{Definition}
 For example, for~$x=\left[-3~0~0~3~4\right]^\top$,
we have $s^-(x)=\sigma\left(\left[-3~3~4  
\right]^\top\right)=1$, whereas  
$s^+(x)= \sigma\left(\left[-3~1~-1~3~4
\right]^\top\right)=3$.

By defintion, 
 \begin{equation}\label{eq:Signineq}
0\leq s^-(x)\leq s^+(x) \leq n-1 \text{ for all } x\in\R^n.
 \end{equation}

Throughout the paper we will employ the following lower semi-continuity property of~$s^-$.
 \begin{Lemma}\citep[Lemma 3.2]{pinkus}\label{lem:cont_s_minus}
 The function $s^- \colon \mathbb{R}^n\to \mathbb{R}$ is lower semi-continuous.
 \end{Lemma}


The notions of weak and strong (number of) sign variations, respectively $s^-$ and $s^+$, yield the following special sets.
\begin{Definition}\label{Def:PkSets}
Given~$k\in [n]$, we define the \emph{set} 
$P^k_- :=\{x\in\R^n\ | \ s^-(x)\leq k-1\}$ (respectively, $P^k_+ :=\{x\in\R^n\ | \ s^+(x)\leq k-1\}$) \emph{of vectors with up to $k-1$ weak} (respectively, \emph{strong}) \emph{sign variations}.
\end{Definition}
For example, for~$k=1$ we have~$P^1_-=\R^n_{\geq 0}\cup \R^n_{\leq 0}$  and $P^1_+=\Int(\R^n_{\geq 0}\cup \R^n_{\leq 0})$. 

We list here several properties of the sets~$P_{-}^k$ and~$P_{+}^k$ that hold for any~$k\in[n]$; the proofs can be found in ~\cite{WEISS_k_posi}.  

\begin{enumerate}
    \item[\textbf{(P1)}]   
    $P^k_{-}$ is closed,  
   $P^k_+$ is open, and~$P^k_+=\Int(P^k_-)$.
    \item[\textbf{(P2)}] 
    If $y\in P^k_{-}$, then $\alpha y \in P^k_{-}$ for any~$\alpha\in\R$; in particular,~$P^k_-$ is a closed cone.
    If $y\in P^k_{+}$, then $\alpha y \in P^k_{+}$ for any~$\alpha\in\R\setminus\{0\}$.
    \end{enumerate}

The set $P^k_- $ is a closed cone, but
unlike~$\R^n_{\geq 0 }$, it is not a convex cone. For example, for~$n=3$, 
$y^1:=\left[-5~1~1\right]^\top $, and $y^2:=\left[1~1~-5\right]^\top 
$, we have~$s^-(y^1)=s^-(y^2)=1$, and hence~$y^1,y^2 \in P^2_-$, but their convex combination
$\frac{1}{2}y^1+\frac{1}{2} y^2=\left[-2~1~-2\right]^\top \not \in P^2_-$.    

In order to state two more properties of~$P^k_-$, recall that a   set~$S\subseteq\R^n$ is called a cone  of rank~$k$ 
if: (i)~$S$ is closed; (ii) $x\in S$ implies that~$\alpha x\in S$ for all~$\alpha\in \R$; and (iii)~$S$ contains a linear subspace of dimension~$k$, and no linear subspace of dimension larger than~$k$. A cone of rank~$k$ is called $k$-solid if there exists a   linear subspace~$V$ of dimension~$k$ such that~$V\setminus\{0\}\subseteq \Int(S)$ (see \cite{KLS89,Sanchez2009ConesOR}). 

 \begin{enumerate}
    \item[\textbf{(P3)}]   $P^k_{-}$ is a cone of rank $k$.
    \item[\textbf{(P4)}]  $P^k_-$ is $k$-solid. 
\end{enumerate}

Fix  a matrix~$A\in\R^{n\times n}$ with positive   entries. 
By the Perron-Frobenius theory (see, e.g.,~\cite{matrx_ana}),  $A $  has special spectral properties. For example, $A$ admits a positive eigenvector~$v$ and thus maps the one-dimensional 
linear subspace~$W^1:=\spanop(v)$ to itself.  The matrix~$A$ also maps the cone of rank one~$P^1_-$ to itself, and maps~$P^1_-\setminus\{0\}$ to~$P^1_+= \Int(P^1_-)$.  
More generally, 
linear operators that map a cone of rank~$k$ to its interior induce a decomposition of~$\mathbb{R}^n$ into a direct sum of invariant subspaces of prescribed dimension, and satisfy a spectral gap condition, as shown by the next theorem. 
\begin{Theorem}\citep{fusco_oliva_1991}\label{thm:fusco}
Let~$S\subseteq \R^n$ be a cone of rank~$k$ with a non-empty interior, and suppose that the matrix~$A\in\R^{n\times n} $ maps~$S\setminus \left\{0 \right\}$ to~$\Int(S)$. Then, there exist unique linear subspaces~$W^1,W^2\subset\R^n$ such that:
    \begin{enumerate}
                \item $W^1 \cap W^2=\{0\}$,  $W^1 \setminus\{0\} \subseteq \Int(S)$,  $W^2\cap S=\{0\}$;
\item $\dim(W^1)=k$, $\dim(W^2)=n-k$; 
        \item $AW^1\subseteq W^1$ and $AW^2\subseteq W^2$.
     \end{enumerate}
    Furthermore, consider the restriction of the linear operator $A$ to the subspace $W^i$, denoted as $A\vert_{W^i}$, for $i=1,2$. Any~$\lambda\in\sigma_1(A)$ and~$\mu \in \sigma_2(A)$, where~$\sigma_i(A)$ is the spectrum of~$A\vert_{W^i}$, 
      satisfy the spectral gap condition
      $|\lambda|>|\mu|$.
\end{Theorem}

The next example demonstrates  that Theorem~\ref{thm:fusco} is a generalization  of the Perron-Frobenius theorem.
\begin{Example}
Fix a matrix~$A\in\R^{n\times n}$ 
with positive entries. 
Let~$S:=P^1_- \subset \mathbb{R}^n$. This is a cone of rank~$1$  with a non-empty interior, and~$A$ maps~$S\setminus\{0\}$ to~$\Int(S)=P^1_+  $.  
 Applying  Theorem~\ref{thm:fusco} implies that~$\R^n$
can be partitioned into  a one-dimensional subspace~$W^1$, such that~$W^1\setminus\{0\} \subseteq P^1_+$,
and an~$(n-1)$-dimensional
subspace~$W^2$.
Hence,~$W^1=\spanop(v)$, where~$v$ is a vector whose entries are all positive. Furthermore,  since~$A$ maps~$W^1$ to~$W^1$, we have that~$Av =\lambda v$ with~$\lambda>0$, and the spectral gap condition in Theorem~\ref{thm:fusco} implies that~$\lambda>|\mu|$ for any other eigenvalue~$\mu$ of~$A$. 
The fact that~$W^1=\spanop(v)$ is unique  implies that
  the eigenvector~$v$ is unique, up to multiplication by a scalar.
This recovers the Perron theorem. Now, if~$A$ is an irreducible matrix with non-negative entries, then~$\exp(A)$ is a matrix with positive entries and the analysis can be extended to this case as well.
\end{Example}

\subsection{$k$-positive linear ODE systems}
Let us recall the notion of (strong) $k$-positivity for linear time-varying ordinary-differential-equation (ODE) systems.
\begin{Definition}\label{def:k_posi}
The linear time-varying~(LTV) system
\begin{align}\label{eq:LTV}
    \dot{x}(t)&=A(t)x(t), \\
    x(t_0)&=x_0,\nonumber
\end{align}
  with fundamental solution matrix $T(t,t_0)$,
is called~\emph{$k$-positive} if $T(t,t_0)P_-^k \subseteq P_-^k$ for all $t\geq t_0$, namely, the flow induced by \eqref{eq:LTV} maps~$P^k_-$ to itself. The system is called \emph{strongly $k$-positive}  if $T(t,t_0)\left(P^k_-\setminus\{0\}\right) \subseteq P_+^k$ for all $t>t_0$, namely, the flow induced by \eqref{eq:LTV} maps~$P^k_-\setminus\{0\}$ to~$P^k_+=\operatorname{int}(P^k_-)$.
\end{Definition}
 Suppose that   $A(t)$ in~\eqref{eq:LTV} is   piecewise continuous. Then the conditions for~$k$-positivity of the LTV system~\eqref{eq:LTV} depend only on the sign structure of~$A(t)$. More precisely, 
 the system~\eqref{eq:LTV} is~$k$-positive, with $k=1,2$, if and only if, for all~$t\geq t_0$, $A(t)$ has the sign pattern $\bar{A}_k$, where
\begin{equation}\label{eq:sign_2_posi}
 \bar{A}_1:= \begin{bmatrix}
    *&\geq 0&\geq 0&\dots&\geq 0&\geq 0\\
     \geq 0&*  &\geq 0&\dots&\geq 0&\geq 0\\
     \geq 0&\geq 0  &*&\dots&\geq 0&\geq 0\\
    \vdots & \vdots & \vdots & \ddots & \vdots &\vdots\\
    \geq 0&\geq 0&\geq 0&\dots&*&\geq 0\\
    \geq 0&\geq 0&\geq 0&\dots&\geq 0&*
\end{bmatrix} \quad \mbox{and} \quad \bar{A}_2:=  \begin{bmatrix}
    * & \geq 0& 0&\dots  &0&\leq 0  \\
    \geq 0& * & \geq 0&\dots&0&0\\
    0& \geq 0 & *& \dots&0&0\\
    \vdots & \vdots & \vdots & \ddots & \vdots &\vdots\\
     0& 0 &   0& \dots&*&\geq 0\\
    \leq 0& 0 &0&    \dots&\geq 0&*
\end{bmatrix}.
\end{equation}
System \eqref{eq:LTV} is \emph{strongly}~$2$-positive if, in addition, matrix~$A(t)$ is irreducible almost everywhere~\citep{WEISS_k_posi} (see also~\citep{ARCAK_D_STAB}). 

The sign pattern $\bar{A}_1$ corresponds to a Metzler matrix, while a matrix with sign pattern $\bar{A}_2$ is not Metzler in general. 
A sign pattern similar to $\bar{A}_2$ in~\eqref{eq:sign_2_posi}
 appears in the seminal work of   \cite{poin_cyclic}   on monotone cyclic feedback systems (see also~\citep{cyclic_sign_vari,FENG2021858,wang_cyclic_feedback}, and the references therein for more recent results).

\begin{Remark}
Systems of the form $\dot x(t)=A x(t)$ with~$A\in\bar A_2 $  often arise in the field of  multi-agent systems~\citep{wooldridge2009introduction}, where the agents $\left\{x_i\right\}_{i=1}^n$ are interconnected through a ring topology with bidirectional links. 
Each agent $x_i$, $i\in\{2,\dots,n-1\}$ receives arbitrarily signed feedback from itself, and non-negative feedback from its two neighbors~$x_{i-1}$ and~$x_{i+1}$. Agent~$x_1$  receives  arbitrary signed feedback from itself,   non-negative feedback from its right neighbor~$x_2$ and non-positive feedback from its ``cyclic left'' neighbor~$x_n$. Agent~$x_n$  receives  arbitrary signed feedback from itself,   non-negative feedback from its left neighbor~$x_{n-1}$ and non-positive feedback from its ``cyclic right'' neighbor~$x_1$.
Since~$A$ may include negative off-diagonal entries, the system~$\dot x=Ax$ is in general not cooperative. 
\end{Remark}

We note in passing that the analysis of sign patterns guaranteeing~$k$-positivity relies on the theory of compound matrices (see, e.g.,~\citep{comp_long_tutorial}), a fundamental tool also for $k$-contractive~\citep{kordercont}, 
$\alpha$-contractive~\citep{Hausdorff_contract}, and totally positive differential systems~\citep{fulltppaper}. 

\subsection{$k$-cooperative nonlinear ODE systems}
Consider the time-invariant nonlinear ODE system
\begin{equation}\label{eq:nonlinear}
 \dot x=f(x),\quad t\geq 0, 
\end{equation}
and assume that its  solutions evolve on a convex state space $\Omega\subseteq\R^n$. Assume  also that there exists a~$\varepsilon>0$ such that~$f\in C^1(\Omega_{\varepsilon})$, where~$\Omega_{\varepsilon}: = \Omega+B(0,\varepsilon)$,  is an $\varepsilon$-neighborhood of $\Omega$, and that for all initial conditions $a\in\Omega$ the system admits a unique solution~$x(t,a)\in\Omega$ for all~$t\geq 0$. Denote the Jacobian of the vector field  by $J_f(x):=\frac{\partial}{\partial x}f(x)$. 

Given two initial conditions~$a,b\in\Omega$, let~$z(t):=x(t,a)-x(t,b)$. Then 
\begin{equation}\label{eq:variational}
  \dot{z}(t)=M_{a,b}(t)z(t),
 \end{equation}
where 
\[
M_{a,b}(t):= \int_0^1 J_f\big( rx(t,a)+(1-r)x(t,b) \big ) \mathrm{d}r.
\]
 The linear time-varying (LTV) system~\eqref{eq:variational} is the  variational equation associated with system~\eqref{eq:nonlinear}. If~$J_f(x)$ has a given sign pattern for all~$x\in\Omega$, then~$M_{a,b}(t)$ has the same sign pattern for all~$t\geq 0$ and for all~$a,b\in \Omega$, because sign patterns are preserved under summation and thus integration. 
\begin{Definition}\label{def:non_k_coop} \citep{WEISS_k_posi}
The nonlinear system \eqref{eq:nonlinear} is  called (\emph{strongly}) \emph{$k$-cooperative} if the associated variational equation \eqref{eq:variational} is (strongly)
$k$-positive for all $a,b\in\Omega$ and all $t\geq 0$.  
\end{Definition}
For example, the nonlinear system \eqref{eq:nonlinear} is $1$-cooperative if the variational  equation \eqref{eq:variational} is $1$-positive, i.e.,   if $J_f(x)$ is Metzler for all $x\in \Omega$. Thus, $1$-cooperativity  is exactly cooperativity. 

\begin{Remark} \label{rem:with_zero1}
Suppose that    the system \eqref{eq:nonlinear}
is strongly $k$-cooperative, and that~$0\in \Omega$ is an equilibrium
of~\eqref{eq:nonlinear}. Fix an initial condition~$a\in\Omega\setminus\{0\}$, and let~$z(t):=x(t,a)-x(t,0)=x(t,a)$.    Considering equation~\eqref{eq:variational}, Definitions~\ref{def:k_posi}
and~\ref{def:non_k_coop} imply  that if~$s^-(x(t,a))\leq k-1$ at some time~$t\geq 0$, then~$s^+(x(\tau,a))\leq k-1$ for all~$\tau>t$. In particular,   $P^k_-$, which is a cone of rank~$k$,  is invariant under the dynamics of the nonlinear system~\eqref{eq:nonlinear}.
\end{Remark}

For  an initial condition $a\in\Omega$, let $\omega(a)$ denote the $\omega$-limit set of $a$,
namely, the set of all points~$y \in \Omega$ for which there exists a sequence of times~$0\leq t_1<t_2<\dots$, with $\lim_{n \to \infty} t_n = +\infty$, such that $\lim_{n \to \infty} x(t_n,a) = y$; see e.g. \cite[p. 193]{diff_eqn_limit_sets}. 
This concept is fundamental when defining the \Poincare-Bendixson  property.

\begin{Definition}\label{Def:PBProp}
Consider the dynamical system \eqref{eq:nonlinear} and let $\mathcal{E}$ denote its set of equilibria. System \eqref{eq:nonlinear} satisfies the 
\emph{strong \Poincare-Bendixson  property}
if,  for any bounded solution $x(t,a)$, with $a\in\Omega$, it holds that 
\begin{equation*}
\omega(a) \cap \mathcal{E}=\emptyset \implies \omega(a) \ \text{is a periodic orbit}.
\end{equation*} 
\end{Definition}
This property is well known to hold for autonomous planar dynamical systems.
Also strongly $2$-cooperative systems satisfy the  strong \Poincare-Bendixson  property~\citep{WEISS_k_posi}; therefore, establishing strong $2$-cooperativity 
provides important information on the asymptotic behavior of the nonlinear system \eqref{eq:nonlinear}.

\section{Main result}\label{sec:main_res}

We provide a sufficient condition for  the existence of  at least one (non-trivial) periodic orbit for strongly $2$-cooperative nonlinear systems. Furthermore, we explicitly  characterize a positive-measure
set of initial conditions such that  all solutions emanating from this set converge to a (non-trivial) periodic orbit. 
\begin{Theorem}\label{thm:main_n_dim}
Consider a strongly 2-cooperative nonlinear time-invariant system
\be\label{eq:nonlin3}
\dot x=f(x), \quad t\geq 0,
\ee
with $f:\R^n\to\R^n$. Let $\Omega\subset \mathbb{R}^n$ be an open, bounded and convex such that 
$f\in C^2(\Omega_{\varepsilon})$, and such that for any
initial condition~$a\in\Omega$ the system admits a unique solution~$x(t,a)\in\Omega$ for all~$t\geq 0$. Suppose that~$0\in \Omega$ is a unique equilibrium of~\eqref{eq:nonlin3} in $\Omega$, and  that $J_f(0) = \frac{\partial}{\partial x}f(0)$ has at least two eigenvalues with a positive real part. Partition $\Omega$ as the   disjoint union 
\[
\Omega=\Omega_{\leq 1} \uplus \Omega_{\geq 2},
\]
where  
\begin{equation}\label{eq:OmegaDecomp}
\begin{array}{lll}
     \Omega_{\leq 1} := \Omega \cap P^2_-  \text{ and }   \Omega_{\geq 2}:= \Omega \setminus P^2_-.
\end{array}
\end{equation}
Then, for any~$a\in \Omega_{\leq 1}\setminus \left\{0 \right\}$,  the solution $x(t,a)$ of \eqref{eq:nonlin3} converges to a (non-trivial)  periodic orbit as $t\to \infty$.
\end{Theorem}

Intuitively speaking, $ \Omega_{\leq 1} = \Omega \cap P^2_- $ is the part of~$\Omega$ with a ``small'' number of sign variations, whereas~$\Omega_{\geq 2}:= \Omega \setminus P^2_-$ is the part with a ``large'' number of sign variations.  The proof of  Theorem~\ref{thm:main_n_dim} is based on showing  that solutions emanating from~$\Omega_{\leq 1}$ do not   converge to the (unique) equilibrium, whence are guaranteed to converge to a periodic orbit, due to the strong \Poincare-Bendixson  property.

The assumption that the equilibrium~$e$ is at the origin is not restrictive, and can be achieved by a coordinate shift.
If $e\neq 0$,
the statement of Theorem~\ref{thm:main_n_dim} becomes:

\vspace{0.05cm}

\emph{Suppose that~$e\in \Omega$, $e \neq 0$, is a unique equilibrium of~\eqref{eq:nonlin3} in $\Omega$, and  that $J_f(e) := \frac{\partial}{\partial x}f(e)$ has at least two eigenvalues with a positive real part. Then, for any $a\in \Omega\setminus \{ e\}$ such that~$s^-(a-e)\leq 1$,
the solution $x(t,a)$ of \eqref{eq:nonlin3} converges to a (non-trivial)  periodic orbit as~$t\to \infty$.
}

\vspace{0.05cm}

It is important to emphasize that Theorem~\ref{thm:main_n_dim} offers an \textit{explicit description}, in terms of sign variations, of initial conditions yielding convergence to a (non-trivial)  periodic orbit. 

\section{Proof of Theorem~\ref{thm:main_n_dim}}\label{sec:proof}
The proof requires an auxiliary result describing  some spectral properties of strongly $2$-cooperative matrices
that may be of independent interest.

\begin{Proposition}\label{lem:hurwitz_and_2_coop}
Assume that  the system $\dot x=Ax$, with $A\in\R^{n\times n}$ and~$n\geq 3$, 
is strongly $2$-positive. Then there exist unique subspaces~$W^1,W^2\subset\R^n$ such that:   
\begin{enumerate}
    \item  
  $W^1 \cap W^2=\{0\}$,\ $W^1\setminus\{0\} \subseteq P^2_+ $,\  $W^2\cap P^2_-=\{0\}$;
\item $\dim(W^1)=2$, $\dim(W^2)=n-2$; 
        \item $AW^1\subseteq W^1$ and $AW^2 \subseteq W^2$.\label{enum:aww}
\end{enumerate}
Furthermore,   order  the eigenvalues of $A$ such that complex conjugate eigenvalues appear consecutively (including multiplicities), and  
\begin{equation}\label{eq:order_eig_stab}
\real(\lambda_1) \geq\dots\geq \real(\lambda_n)  .  
\end{equation} 
Then 
\begin{equation}\label{eq:gap}
\real(\lambda_2)>\real(\lambda_3),
\end{equation}
and $\sigma(A\vert_{W^1}) = \left\{\lambda_1 ,\lambda_2\right\}$.
\end{Proposition}
\begin{pf}
Fix $s>0$.
The eigenvalues of
  the matrix $B_s:=\exp(As)$
  are $\exp(\lambda_i s)$, $i\in[n]$, and~\eqref{eq:order_eig_stab} implies that
\begin{equation}\label{eq:ExpEigOrder}
|\exp(\lambda_1 s)|
\geq |\exp(\lambda_2 s)|
\geq \dots\geq |\exp(\lambda_n s )|,    
\end{equation}
with complex conjugate eigenvalues appearing in consecutive pairs. Since $\dot{x}=Ax$ is strongly $2$-positive, the matrix~$B_s$ (which is the fundamental solution matrix  of $\dot{x}=Ax$)  maps~$P^2_-\setminus\{0\}$ to~$\Int(P^2_-) = P^2_+$. Since~$P^2_-$ is a cone of rank~$2$, applying Theorem~\ref{thm:fusco} yields two \emph{unique} linear subspaces~$W^i(s),\ i=1,2$, which satisfy properties~(1) and~(2) stated in the theorem. Furthermore, $B_s W^i(s)\subseteq W^i(s),\ i=1,2$, and we have the spectral gap condition
\begin{equation}\label{eq:SpecGapExp}
\exp(\real(\lambda_2)s)=|\exp(\lambda_2 s)|
>|\exp(\lambda_3 s)|=\exp(\real(\lambda_3)s),    
\end{equation}
therefore~\eqref{eq:gap} holds.

We now show that the linear subspaces $W^i(s),\ i=1,2$, do not depend on~$s$. It is enough to prove this for~$W^1(s)$, as the proof for~$W^2(s)$ is identical. Define
the sequence~$T_n := 2^{-n}$, $n=1,2,\dots$,
and  let~$W^1_n:=W^1(T_n)$,   denote the corresponding linear subspace. For any $n\geq 1$, we have 
\begin{equation*}
    \exp(A T_n) W^{1}_{n+1} =  \exp(2A T_{n+1}) W^{1}_{n+1}
    = \exp( AT_{n+1})    \exp( AT_{n+1})
    W^1_{n+1}
    \subseteq W^1_{n+1}, 
\end{equation*}
where we used the fact that~$\exp(AT_{n+1}) W^1_{n+1}\subseteq W^1_{n+1}$. Since $W^i_{n+1},\ i=1,2$, also satisfy conditions~(1) and~(2) in the proposition, the uniqueness of $W^1_{n}$ implies that~$ W^1_{n} =W^1_{n+1} $,  for all $n\geq 1$, so we can simply write~$W^1:=W^1_n$.

Fix~$T>0$ and  $\zeta\in W^1$. For any $\epsilon>0$, there exists
a finite sum of elements of $\left\{T_n \suchthat  n=1,2,\dots \right\}$ (with possible repetitions) in the form~$  \sum_{n=1}^{N(\epsilon)} a(\epsilon,n) T_{n}$,
where~$a(\epsilon,n)$ are   
integers,    such that 
\[
\left|T-\sum_{n=1}^{N(\epsilon)} a(\epsilon,n) T_{n} \right|<\epsilon.
\]
Then, by the previous step, 
\begin{equation*}
    \exp\left(A \left(\sum_{n=1}^{N(\epsilon)} a(\epsilon,n) T_{n}\right)\right) \zeta = \left(\prod_{n=1}^{N(\epsilon)} \left(\exp(AT_n) \right)^{a( \epsilon,n)}\right)\zeta\in W^1. 
\end{equation*}
Employing continuity of the  mapping~$t \mapsto \exp(At)\zeta$, we conclude that $\exp(AT)\zeta \in W^1$, so $\exp(AT)W^1\subseteq W^1$. Using again uniqueness  of the subspaces decomposing $\mathbb{R}^n$ into a direct sum, and the fact that~$T$ is arbitrary,   we conclude that the subspaces $W^1 ,W^2 $  do not  depend  on~$T$.

Fix~$h>0$. By the previous arguments, $\exp(Ah)\zeta \in W^1 $ and, since~$W^1$ is a linear subspace,
\[
h^{-1}\left(e^{Ah}\zeta - \zeta \right)\in  W^ 1.
\]
Taking the limit as~$h\downarrow 0$, we conclude that $\dot \zeta = A\zeta\in W^1$, whence $W^1$ is~$A$-invariant; a similar argument shows that~$W^2$ is also $A$-invariant.  In particular, $\mathbb{R}^n $ can be decomposed as the direct sum   of $A$-invariant subspaces~$W^1\bigoplus W^2$, whence~$A$ can be represented as~$A = \operatorname{diag}\left(A\vert_{W^1}, A\vert_{W^2} \right)$. Recalling \eqref{eq:order_eig_stab} and \eqref{eq:gap}, we must have~$\lambda_1,\lambda_2\in \sigma(A\vert_{W^1})$. Indeed, the representation of $A$ yields $\sigma(A) = \sigma(A\vert_{W^1}) \bigcup \sigma(A\vert_{W^2})$. If, for some $i=1,2$, it holds that  $\lambda_i\in \sigma(A\vert_{W^2})$, then $\operatorname{exp}(\lambda_is)\in \sigma \left(B_s \vert_{W^2} \right)$ for $s>0$, which contradicts the spectral gap condition for~$B_s$ in~\eqref{eq:ExpEigOrder} and~\eqref{eq:SpecGapExp}. Since $\operatorname{dim}\left(W^1\right)=2$, we conclude that $\sigma(A\vert_{W^1}) = \left\{\lambda_1,\lambda_2 \right\}$,
and this completes the proof of Proposition~\ref{lem:hurwitz_and_2_coop}.
\end{pf}

The next result follows from Proposition~\ref{lem:hurwitz_and_2_coop}.
\begin{Corollary}\label{Cor:JordA}
Under the assumptions of Proposition \ref{lem:hurwitz_and_2_coop}, there exists
a non-singular matrix $S\in\R^{n\times n}$ such that 
\begin{equation}\label{eq:ortho_basis}
SAS^{-1} = \operatorname{diag}\left(\Lambda,\Psi \right),   
\end{equation}
where $\Lambda\in\R^{2\times 2}$ is  the \emph{real} Jordan form of $A\vert_{W^1}$, and $\Psi\in\R^{(n-2)\times(n-2)}$.
Furthermore,~$\Lambda$ 
 has  one of the following three forms:
 \begin{align*}
(i)\ & \Lambda=\begin{bmatrix}
u_1 & 0\\ 0 & u_2
\end{bmatrix},\ u_1,u_2\in \mathbb{R}, \text{ with } u_1\geq u_2; \\
(ii)\ & \Lambda=\begin{bmatrix}
u & -v \\
v & u
\end{bmatrix},\ u,v\in \mathbb{R}, \text{ with } v\neq 0; \\
(iii)\ & \Lambda=\begin{bmatrix}
    u & 1\\
    0 & u
\end{bmatrix},\ u\in \mathbb{R}. 
\end{align*}
\end{Corollary} 
The three possible forms in the corollary correspond to the cases:
 $(i)$~$\lambda_1=u_1$ and $\lambda_2=u_2$ are real and distinct eigenvalues and each has  geometric multiplicity 1,  or 
 $\lambda_1=\lambda_2$ and has geometric multiplicity 2; 
 $(ii)$ $\lambda_1=u+iv$ and~$\lambda_2=\overline{\lambda}_1$ are two complex conjugate eigenvalues; and
 $(iii)$~$\lambda_1=\lambda_2=u$ is a real eigenvalue of algebraic multiplicity~$2$ and geometric multiplicity~$1$.

 We can now prove Theorem~\ref{thm:main_n_dim}. 
\begin{pf}
Since the system \eqref{eq:nonlin3} is  strongly 2-cooperative,  we have that 
\begin{equation}\label{eq:imilz'}
x(t,a)\in P^2_-\setminus \left\{0 \right\} \text { for some } t\geq 0 \Rightarrow x(\tau,a)\in P^2_+ =\Int(P^2_-) \text{ for all } \tau>t.
\end{equation}
Combining this with the assumed invariance of~$\Omega$ implies that~$\Omega_{\leq 1}$ is also invariant.  Note that $0\in \Omega_{\leq 1}$. To prove the theorem we use the   strong \Poincare-Bendixon property of strongly 2-cooperative systems, the instability of~$J_f(0)$ and the spectral properties of strongly 2-positive systems in Proposition \ref{lem:hurwitz_and_2_coop} to
 show that, for all~$a\in \Omega_{\leq 1}\setminus\left\{0 \right\}$, the solution $x(t,a)$  does not converge to~$0$ and thus it must converge to a periodic solution. The proof includes 
 several steps.

\underline{\emph{Step 1:}} A change of variables.

Let~$A:=J_f(0)$. Since the system \eqref{eq:nonlin3} is  strongly 2-cooperative,  the linear system~$\dot{y}=Ay$ is strongly $2$-positive. Let~$S\in\R^{n\times n}$ be the invertible matrix in  Corollary~\ref{Cor:JordA}.
For $\delta>0$, let
\begin{equation}\label{eq:DiagScal}
S_{\delta} := \left[
\begin{array}{c|c}
\begin{matrix} 1 &0\\0&\delta \end{matrix} & \textbf{0}\\ \hline
\textbf{0} & I_{n-2}
\end{array}
\right]
 S,
\end{equation}
where $\textbf{0}$ denotes an all-zero matrix of the appropriate size.
Then~$S_\delta$ is invertible. Introduce the change of coordinates~$q (t) : = S_{\delta}x(t)$. Then  
\begin{equation}\label{eq:q_sys} 
    \dot{q} =g(q),  \text{ with } g(q):= S_{\delta} f(S_{\delta}^{-1}q).
\end{equation}
The state-space of this system is   the open, bounded and convex set~$\tilde{\Omega}:=S_\delta \Omega$. 
Let~$\tilde{\Omega}_{\leq 1}:=S_{\delta}\Omega_{\leq 1}$ and~$\tilde{\Omega}_{\geq 2}:= S_{\delta}\Omega_{\geq 2}$.
Then~$\tilde{\Omega}= \tilde{\Omega}_{\leq 1} \uplus \tilde{\Omega}_{\geq 2} $, with $\tilde{\Omega}_{\leq 1}$ being invariant under~\eqref{eq:q_sys}. In these new coordinates, we have
\be\label{eq:tilde_space}
\tilde{W}^1:=S_{\delta}W^1=\operatorname{span}\left\{e^1,e^2 \right\}, \quad
\tilde{W}^2:=S_{\delta}W^2=\operatorname{span}\left\{e^3,\dots,e^n \right\},
\ee
where $\left\{e^i \right\}_{i=1}^n$ is the standard basis in $\mathbb{R}^n$.
Also, 
\begin{enumerate}
    \item  
  $\tilde{W}^1 \cap \tilde{W}^2=\{0\}$,\ $\tilde{W}^1\setminus\{0\} \subseteq S_{\delta}P^2_+ $,\  $\tilde{W}^2\cap S_{\delta}P^2_-=\{0\}$;
\item $\dim(\tilde{W}^1)=2$, $\dim(\tilde{W}^2)=n-2$.
\end{enumerate}
  The origin is an equilibrium of the system~\eqref{eq:q_sys},
and its Jacobian    is
\begin{equation*}
 J_{g}(z)= S_\delta J_f(S_\delta ^{-1} z) S_{\delta}^{-1}.   
\end{equation*}
Using~\eqref{eq:ortho_basis},  \eqref{eq:DiagScal}, and the assumption that~$J_f(0)$ has at least two eigenvalues with a positive real part gives 
\begin{equation}\label{eq:TildA}
\tilde{A}_\delta :=  J_{g}(0) = S_{\delta}A S_{\delta}^{-1} =\left[
\begin{array}{c|c}
    \tilde \Lambda_\delta
    & 0\\ \hline
    0 & \Psi
\end{array}
\right] , 
\end{equation}
with
\[ 
\tilde \Lambda_\delta:= 
    \operatorname{diag}\left(1,\delta \right)\Lambda \operatorname{diag}\left(1,\delta^{-1} \right) ,  
\]
and~$\tilde \Lambda_\delta$ has one of the following three forms (see  Corollary \ref{Cor:JordA}):
\begin{equation}\label{eq:CasesAtild}
(i)\ \tilde \Lambda_\delta=\begin{bmatrix}
u_1 & 0\\ 0 & u_2
\end{bmatrix},\ u_1,u_2>0; \quad (ii)\ \tilde \Lambda_\delta=\begin{bmatrix}
u & -{\delta}^{-1} v \\
\delta v & u
\end{bmatrix},\ u>0, v\neq 0; \quad (iii)\ \tilde \Lambda_\delta=\begin{bmatrix}
    u & {\delta}^{-1}\\
    0 & u
\end{bmatrix},\ u>0 . 
\end{equation}
Let~$\tilde P_\delta: = (\tilde \Lambda_\delta+\tilde \Lambda_\delta^\top)/2$. If case~$(i)$ holds then $\tilde P_\delta$
 is in fact independent of~$\delta$ and is positive-definite. If~case~$(ii)$ holds then~$\tilde P_\delta$ is positive-definite for~$\delta=1$. If case~$(iii)$ holds then $\tilde \Lambda_\delta$ is positive-definite for any~$\delta>0$ sufficiently large. Summarizing, we can always choose~$\delta>0$
such that~$\tilde P_\delta$ is positive-definite.

\underline{\emph{Step 2:}} Derivation of an appropriate conic neighborhood around $\tilde{W}^2$.

Recall that $\tilde{W}^2\cap S_{\delta}P^2_- = \left\{0 \right\}$. Hence, $\tilde{W}^2\cap \tilde{\Omega}\subseteq \tilde{\Omega}_{\geq 2}$, with $\tilde{\Omega}_{\geq 2}$ being an open set. 
For any~$\xi \in \mathbb{R}^n\setminus \left\{0 \right\}$, define
\begin{equation*}
    {p}(\xi): = \frac{\sqrt{\sum_{i=3}^n\xi_i^2}}{\left|\xi \right|},
\end{equation*}
that is, $p(\xi)$ is the ratio between the norm of the projection of $\xi$ on~$\tilde{W}^2$ and the norm of $\xi$. By definition,~$p$ is a homogeneous function of degree zero,  
that is, 
\be\label{eq:p_is_jomog}
p(\alpha \xi)=p(\xi) \text{ for any } \alpha\in\R\setminus\{0\},
\ee
and also~$p(\xi)\in [0,1]$ for all~$\xi \in \mathbb{R}^n\setminus \left\{0 \right\}$. By~\eqref{eq:tilde_space},  for any~$\xi\in \mathbb{R}^n\setminus \left\{0\right\}$ we have 
\be\label{eq:pisone}
p(\xi)=1 \text{ if and only if } \xi \in \tilde W^2\setminus\{0\}.
\ee

The following lemma  shows that there is a conic neighborhood around $\tilde{W}^2\setminus\left\{0\right\}$, whose intersection with $\tilde{\Omega}$ lies entirely in $\tilde{\Omega}_{\geq 2}$. 
\begin{Lemma}\label{lemm:AngularSep}
There exists   $ \tilde \varepsilon \in (0,1)$ such that, for any $\xi \in \tilde{\Omega}\setminus \left\{0 \right\}$, we have that
\begin{equation*}
p(\xi) > 1-\tilde \varepsilon \Rightarrow \xi \in \tilde{\Omega}_{\geq 2}.
\end{equation*}
\end{Lemma}

\begin{pf}
We prove the statement by contradiction. Assume that the claim  does not hold. Then    there exist sequences~$\left\{ \xi^\ell  \right\}_{\ell = 1}^{\infty}\subseteq \tilde{\Omega}\setminus \left\{0 \right\} $ 
and~$\left\{\varepsilon_{\ell} \right\}_{\ell=1}^{\infty}\subseteq (0,1)$ such that 
\begin{equation}\label{eq:AssumpContr}
\lim_{\ell \to \infty}\varepsilon_{\ell} = 0,\qquad  p\left(\xi^{\ell} \right)  > 1-\varepsilon_{\ell},\qquad  \xi^{\ell}\in \tilde{\Omega}_{\leq 1},
\end{equation}
where we used the fact 
that~$\tilde{\Omega}= \tilde{\Omega}_{\leq 1} \uplus \tilde{\Omega}_{\geq 2} $. 
  Let $\mu>0$ be such that $B(0,\mu)\subseteq \tilde{\Omega}$. Using the fact that~$p$ is homogeneous of degree zero and~$S_{\delta}P^2_-$ is scaling invariant (see Property \textbf{(P2)} in Section~\ref{sec:p2_homog}), we can scale the vectors~$\left\{ \xi_{\ell}\right\}_{\ell=1}^{\infty}$ to~$\partial B(0,\mu)$, while preserving \eqref{eq:AssumpContr}, to obtain 
\begin{equation}\label{eq:scaledxi}
\lim_{\ell \to \infty}\varepsilon_{\ell} = 0,\qquad  p\left(\mu \frac{\xi^{\ell}}{|\xi^{\ell}|} \right)  > 1-\varepsilon_{\ell},\qquad  \mu \frac{ \xi^{\ell}} {|\xi^{\ell}|} \in \tilde{\Omega}_{\leq 1}.
\end{equation}
Passing to a sub-sequence, if needed, we may assume that~$ \mu \frac{ \xi^{\ell}} {|\xi^{\ell}|}$ converges to a limit~$\xi$, with~$|\xi|=\mu>0$, and~$p(\xi)\geq 1$, so~$p(\xi) =  1$, whence $\xi \in \tilde W^2 \setminus\{0\}$ in view of~\eqref{eq:pisone}. 
Since~$P^2_-$ is closed and~$S_\delta$ is invertible, $S_\delta P^2_- $ is closed, so~\eqref{eq:scaledxi} implies that~$\xi\in S_\delta P^2_- $. 
Using~\eqref{eq:pisone} gives
\[
 \xi \in (S_\delta P^2_-)\cap  (\tilde W^2 \setminus\{0\}).  
\]
  However, the intersection of these two sets is empty. This contradiction completes the proof of Lemma~\ref{lemm:AngularSep}.  
 \end{pf}

\begin{Corollary}\label{Cor:CoordWeights}
There exists 
$\tilde \theta>0$ such that 
\begin{equation}\label{eq:tilde_q}
\xi \in \tilde{\Omega}_{\leq 1}\setminus \left\{0 \right\}   \implies
\left|\xi \right|^2 \leq   \left(\xi_1^2+\xi_2^2 \right) \tilde \theta.
\end{equation}
\end{Corollary}

\begin{pf}
Let~$\tilde\epsilon
\in(0,1)$ be as in Lemma~\ref{lemm:AngularSep},
and define~$\tilde q:= (1-\tilde \epsilon)^2$, so~$\tilde q\in(0,1)$.
Fix~$\xi \in \tilde{\Omega}_{\leq 1}\setminus \left\{0 \right\}$. Then
Lemma~\ref{lemm:AngularSep} gives
$
(p(\xi))^2 \leq (1-\tilde \varepsilon)^2=\tilde q,
$
that is, 
$\sum_{i=3}^n \xi_i^2 \leq \tilde q \sum_{i=1}^n\xi_i^2  $,
so
 $(1-\tilde q)\sum_{i=3}^n\xi_i^2 \leq \tilde q\left(\xi_1^2+\xi_2^2 \right)$. Thus,
 \begin{equation*}
     |\xi|^2= \xi_1^2+\xi_2^2+\sum_{i=3}^n \xi_i^2
     \leq \left(1+\frac{\tilde q}{1-\tilde q} \right)
     \left(\xi_1^2+\xi_2^2 \right)
     =\frac1{1-\tilde q} \left(\xi_1^2+\xi_2^2 \right),
 \end{equation*}
 and~\eqref{eq:tilde_q} holds for~$\tilde \theta:=\frac1{1-\tilde q}>0$.
 \end{pf}

\underline{\emph{Step 3:}} Local  Lyapunov analysis.\\
We now show that for any $a\in \tilde{\Omega}_{\leq 1}\setminus \left\{0 \right\}$  the solution $q(t,a)$ of \eqref{eq:q_sys}  remains bounded away from the unique equilibrium $0\in \tilde{\Omega}$ of \eqref{eq:q_sys}. Applying a Taylor expansion to the right-hand side of \eqref{eq:q_sys} gives 
\begin{equation}\label{eq:qsystem}
\dot{q} =  \tilde{A} q + h(q).
\end{equation}
Since we assume that~$f$ is~$C^2$ on a convex neighborhood that includes the closure of~$\Omega$ (and this carries over to~$g$ and~$\tilde \Omega$), there exists~$M>0$ such that the nonlinear terms in~\eqref{eq:qsystem} satisfy 
\be\label{eq:tayolr_bound}
|h_i(q)|\leq M|q|^2 \text{ for all } q\in\tilde\Omega, \ i\in[n].
\ee

Let
\begin{equation}
    V(q): = \frac{1}{2}\left(q_1^2+q_2^2\right).
\end{equation}
  Intuitively, the value $ (2V(q))^{1/2}=(q_1^2+q_2^2)^{1/2}$ is the norm of the projection of $q$ on the dominant unstable directions (eigenvectors) of $\tilde{A}_{\delta}$ in \eqref{eq:TildA}. 

For any~$\eta>0$,  
let
\be\label{eq:def_H}
H_{\eta}: = \tilde{\Omega}_{\leq 1}\cap  \left\{q\in \mathbb{R}^n \suchthat  V(q)\geq \eta \right\}.
\ee
  Note further that~$0\not\in H_\eta$. 
Recalling that $\tilde{\Omega}_{\leq 1}$ is invariant for~\eqref{eq:qsystem},
our next goal is to show that there exists a value~$\eta_0>0$
such that~$H_\eta $ is an invariant set of~\eqref{eq:qsystem}  for any $0<\eta<\eta_0$.  For this purpose, we evaluate the Lyapunov derivative~$\dot V(q)$
for $q\in \mathcal{H}_{\eta}$ and show that there exists $\eta_0>0$ such that, if $0<\eta<\eta_0$, then 
\begin{equation*}
V(q) = \eta \Longrightarrow \dot{V}(q)>0.
\end{equation*}
The latter implies, in particular, that 
if~$q(t,a)$ is a solution of~\eqref{eq:qsystem} with~$V(a)\geq \eta$, where~$0<\eta<\eta_0$, then~$\inf_{t\geq 0}V(q(t,a))\geq \eta$.

The Lyapunov derivative is
\begin{equation*}
    \dot V(q)= q _1\dot q_1+q_2\dot q_2= \begin{bmatrix}
        q_1&q_2
    \end{bmatrix}\tilde \Lambda_\delta\begin{bmatrix}
        q_1\\q_2
    \end{bmatrix}
    +\begin{bmatrix}
        q_1&q_2
    \end{bmatrix} \begin{bmatrix}
        h_1(q)\\h_2(q)
    \end{bmatrix}.
\end{equation*}
Recall that we can always choose~$\delta>0$ such that~$(\tilde \Lambda_\delta+\tilde \Lambda_\delta^\top)/2$ is positive-definite, so there exists~$\alpha>0$ such that
\begin{equation}\label{eq:sec_term}
    \dot V(q)
    \geq  \alpha \left|\begin{bmatrix}
        q_1&q_2
    \end{bmatrix}^{\top} \right|^2 +\begin{bmatrix}
        q_1&q_2
    \end{bmatrix} \begin{bmatrix}
        h_1(q)\\h_2(q)
    \end{bmatrix}.
\end{equation}
To bound the second term in the sum, note that 
\begin{equation*}
\left|
\begin{bmatrix}
        q_1&q_2
    \end{bmatrix} \begin{bmatrix}
        h_1(q)\\h_2(q)
    \end{bmatrix}\right| \leq
  \left|
\begin{bmatrix}
        q_1&q_2
    \end{bmatrix}^{\top} \right| \ 
    \left|\begin{bmatrix}
        h_1(q)\\h_2(q)
    \end{bmatrix}\right| 
    \leq  \left|
\begin{bmatrix}
        q_1&q_2
    \end{bmatrix}^{\top} \right| \ 2M |q|^2
    \leq 
    2M \tilde\theta \  \left|
\begin{bmatrix}
        q_1&q_2
    \end{bmatrix}^{\top} \right|^3,
\end{equation*}
where the first inequality follows from the Cauchy-Schwarz inequality, the second inequality uses~\eqref{eq:tayolr_bound}, 
and the third inequality follows from Corollary~\ref{Cor:CoordWeights}. Thus,
\[
\dot V(q)\geq
2\alpha V(q)-2M\tilde\theta \big (2V(q)\big )^{3/2} = 2V(q) \left( \alpha - M\tilde{\theta}V(q)^{\frac{1}{2}}\right).
\]
This implies that there exists a sufficiently small~$\eta_0>0$ such that 
if~$V(q)=\eta$, where $0<\eta<\eta_0$,
then~$\dot V(q)>0$,
so~$H_{\eta}$ is an invariant set of~\eqref{eq:qsystem}. 

\underline{\emph{Step 4:}} Completing  the proof of Theorem \ref{thm:main_n_dim}.

Consider the nonlinear system~\eqref{eq:OmegaDecomp}. Fix an initial condition~$a\in \Omega_{\leq 1}\setminus \left\{0 \right\}$. Let~$\tilde{a}:=S_{\delta}a\neq 0$ (where $\delta>0$ is chosen so that~$\tilde P_\delta$ is positive-definite).   Then~$\tilde a = (\tilde{a}_1,\dots,\tilde{a}_n)^{\top}\in\tilde \Omega_{\leq 1}\setminus\{0\}$.
If~$\tilde{a}_1=
\tilde{a}_2=0$  then
\[
0\neq \tilde a\in \tilde W^2\cap \tilde \Omega_{\leq 1}\subseteq \tilde W^2 \cap S_\delta  P^2_-,
\]
which is a contradiction. 
We conclude 
  that~$\tilde{a}_1^2+\tilde{a}_2^2\neq 0$, so
 there is some $\eta\in(0,\eta_0)$ such that $\tilde{a}\in H_{\eta}$. By
 Step~3, the solution $q(t,a)$ of \eqref{eq:q_sys} remains in~$H_\eta$ and is thus  
 bounded away from~$0\in \tilde{\Omega}$. This implies that~$x(t,a)$ is  bounded away from~$0\in \Omega$, which is the only equilibrium of \eqref{thm:main_n_dim} in~$\Omega$. Since the system is strongly $2$-cooperative, it satisfies the strong \Poincare-Bendixson property in Definition~\ref{Def:PBProp}.   As $0\notin \omega(a)$,  we conclude that $\omega(a)$ is    a periodic orbit and~$x(t,a)$ converges to~$\omega(a)$ as~$t\to \infty$.
This completes the proof of Theorem~\ref{thm:main_n_dim}.
\end{pf}

   \section{Applications}\label{sec:applic}
 We describe two applications of Theorem~\ref{thm:main_n_dim}  to models from systems biology.
 
 \subsection{Convergence to 
 periodic orbits in the $n$-dimensional Goodwin model } \label{subsec:applic}
 
The Goodwin model captures a classical   biochemical circuit design where enzyme or
protein synthesis are regulated by incorporating a negative-feedback of the end product~\citep{GOODWIN1965425}.
This has become a touchstone circuit in systems biology, see the recent survey paper by~\cite{gonze2021} and the many references therein. The model includes~$n$ first-order differential equations: 
\begin{equation}\label{eq:3dgood}
\begin{cases}
\dot x_1& =-\alpha_1 x_1 + \frac{1}{1+x_n^m},\\
\dot x_2& =-\alpha_2 x_2 +x_1,\\
\dot x_3& =-\alpha_3 x_3 + x_2,\\
&\vdots\\
\dot x_n& =-\alpha_ n x_n + x_{n-1}, 
\end{cases}
\end{equation}
where~$\alpha_i> 0$ for all~$i=1,\dots,n$,  
and~$m\in \mathbb{N}$. 
Denote by
$\alpha:=\alpha_1\alpha_2\alpha_3\cdots\alpha_n$
the product of all the dissipation gains 
in the system.

The state space of~\eqref{eq:3dgood} is~$\Omega:=\R^n_{\geq 0}$. Furthermore, all trajectories are bounded. In fact, any trajectory emanating from~$\Omega$ eventually enters into the closed box
  $\B_G := \{x\in\R^n_{\geq 0} 
  \suchthat x_1 \leq  \alpha_1^{-1},\; x_2\leq (\alpha_1\alpha_2)^{-1}, \dots,x_n\leq (\alpha_1\dots\alpha_n)^{-1}\}$.
Moreover, since on $\partial \B_G$ the vector field of~\eqref{eq:3dgood} points into $\operatorname{int}(\B_G)$, the solutions eventually enter $\operatorname{int}(\B_G)$.

Since~$ {\B_G} $ is a compact, convex,  and invariant set,     $\operatorname{int}(\B_G)$ is also an invariant set (see \cite{mcs_angeli_2003}), and there exists at least one equilibrium point~$e\in \operatorname{int}\left(\B_G\right)$. By computing the equilibria of system~\eqref{eq:3dgood}
we can see that~$e_n$ is a real and positive root of the polynomial 
\be\label{eq:ps_e3}
Q(s): = \alpha  s^{m+1} +\alpha  s-1,
\ee
and, in view of Descartes' rule of signs,
there is a unique such~$e_n$. Then
\be\label{eq:eq+good}
e_j = \left(\prod_{k\geq j+1}\alpha_k \right)e_n,\text{ for all } j\in [n-1],
\ee
and hence~$e\in \operatorname{int}\left(\mathcal{B}_{G} \right
)$ is unique.

Several studies (see e.g.~\cite{sanchez_goodein_stab} and the references therein) 
derived conditions guaranteeing that the equilibrium~$e\in\operatorname{int}\left(\mathcal{B}_G\right)$ is globally asymptotically stable.  
\cite{Tyson_3D} analyzed the special case of~\eqref{eq:3dgood} with~$n=3$. He 
noted that if~$e$ is locally  asymptotically stable, then one may expect that all solutions converge to~$e$, and   proved 
 that system~\eqref{eq:3dgood}
admits a periodic solution whenever~$e$ is unstable.  For~$n=3$, the model can also be studied using the theory of competitive dynamical systems~\citep{hlsmith}. The case~$n=3$ has also been analyzed using the theory of Hopf bifurcations~\citep{Woller_2014_3D_Goodwin}. For a general~$n$, the analysis using Hopf bifurcations becomes highly
non-trivial and results exist only for special cases, e. g.  
under the additional assumption that all the~$\alpha_i$'s are equal, see~\cite{hopf_Goodwin_n_dim}.
\cite{HASTINGS1977}   studied the general $n$-dimensional case and proved that, if the Jacobian of the vector field  at the equilibrium has no repeated eigenvalues and at least one eigenvalue with a positive real part, then the system admits a non-trivial periodic orbit; the proof relies on the Brouwer fixed point theorem.

Our Theorem~\ref{thm:main_n_dim} allows us to prove the following result. 
\begin{Corollary}
    \label{coro:good}
    Consider the $n$-dimensional  Goodwin model~\eqref{eq:3dgood} with~$n\geq 3$,  and let~$e$ denote the unique 
    equilibrium
    in~$\operatorname{int}\left(\B_G\right)$. Let~$J:\R^n_{\geq 0}\to\R^{n\times n}$ denote the Jacobian of the vector field of the Goodwin model. Suppose that~$J(e)$
   has at least one eigenvalue with a positive real part.
   Then, for any initial condition~$a\in\R^n_{\geq 0}\setminus\{e\}$
   such that~$s^-(a-e)\leq 1$,
   the solution~$x(t,a)$ of~\eqref{eq:3dgood} converges to a (non-trivial) periodic orbit as $t\to \infty$. 
\end{Corollary}   

\begin{pf}

The Jacobian of~\eqref{eq:3dgood}
\[
J(x)=\begin{bmatrix}
    -\alpha_1 & 0 &0&\dots& 0& -\frac{m x_n^{m-1}}{(1+x_n^m)^2}\\
    1&-\alpha_2&0&\dots&0&0\\
    0 & 1&-\alpha_3&\dots&0&0\\
    \vdots&\vdots&\vdots&\ddots&\vdots&\vdots\\
  0&0&0&  \dots&-\alpha_{n-1}&0\\
  0&0&0&  \dots&1&-\alpha_n
\end{bmatrix}
\]
has the sign pattern $\bar A_2$ in \eqref{eq:sign_2_posi} for all~$x\in\mathbb{R}^n_{\geq 0}$, hence the system is   2-cooperative on~$\mathbb{R}^n_{\geq 0}$. 
We now show that the system is  strongly $2$-cooperative on~$\mathbb{R}^n_{\geq 0}$. 
 If $x(t,a)$ is a solution of \eqref{eq:3dgood} with initial condition~$a \in \mathbb{R}^n_{\geq 0}$   and $x_n(t_0,a )=0$ at some time~$t_0\geq0$, then there exists some $\delta>0$ such that $t\in (t_0,t_0+\delta) \implies x_n(t,a)>0$. In particular, the set $\left\{t\geq 0 \ | \ x_n(t,a)=0 \right\}$ is at most countable, and this implies that the time-varying matrix $M(t)$ in the variational equation \eqref{eq:variational}, which is obtained from integrating~$J(x(t))$, is irreducible for almost all~$t$. As a result, 
 the system is strongly  2-cooperative on~$\mathbb{R}^n_{\geq 0}$ and the set $\left\{x\in \mathbb{R}^n_{\geq 0}\ \suchthat \ s^-(x-e)\leq 1 \right\}$ is forward invariant  (see Remark~\ref{rem:with_zero1}).

To analyze the eigenvalues of~$J(e)$,  
consider the   matrix 
\begin{equation}\label{eq:GoodwinMatrix}
A := \begin{bmatrix}
    -\alpha_1 & 0 &0&\dots& 0& -\beta_n\\
    \beta_1&-\alpha_2&0&\dots&0&0\\
    0 & \beta_2&-\alpha_3&\dots&0&0\\
    \vdots&\vdots&\vdots&\ddots&\vdots&\vdots\\
  0&0&0&  \dots&-\alpha_{n-1}&0\\
  0&0&0&  \dots&\beta_{n-1}&-\alpha_n
\end{bmatrix},
\end{equation}
with $n\geq 3$, and
$
\alpha_i,\beta_j>0 \text{ for all } i,j\in [n].
$
Let $\beta:=\prod_{j=1}^n \beta_j$.
The characteristic polynomial of~$A$ is
\begin{equation}\label{eq:poly_A}
p_A(s):= \beta+\prod_{j=1}^n
(s+\alpha_j)=s^n+a_{n-1}
s^{n-1}+\dots+ a_1 s+a_0,
\end{equation}
and since the~$\alpha_i$'s and~$\beta_j$'s are all positive, the coefficients of $p_A(s)$ satisfy $a_i>0$ for all~$i$.
This implies that~$p_A(s)$ cannot have a real positive zero. Since we assume that~$J(e)$ has an unstable eigenvalue~$\lambda_1$, we conclude that~$\lambda_1$ is not real, and therefore its complex conjugate~$\bar \lambda_1$ is another eigenvalue of~$J(e)$ with a positive real part, so~$J(e)$ has at least two unstable eigenvalues. 

Take now $a\in \mathbb{R}^n_{\geq 0}$ such that $s^{-}(a-e)\leq 1$ and consider $x(t,a)$. There exists a time~$T(a)>0$ such that $x\left(T(a),a \right)\in \operatorname{int}\left(\B_G \right)$. Furthermore, as the set $\left\{x\in \mathbb{R}^n_{\geq 0}\ \suchthat \ s^-(x-e)\leq 1 \right\}$ is forward invariant, we also have $s^-\left(x\left( T(a),a\right)-e \right)\leq 1$. Since $\omega(a) = \omega\left(x\left(T(a),a \right) \right)$,  applying Theorem~\ref{thm:main_n_dim} in $\operatorname{int}\left(\mathcal{B}_G \right)$ completes the proof of Corollary~\ref{coro:good}.
\end{pf}

The next  numerical example demonstrates 
Corollary~\ref{coro:good}.
\begin{Example}\label{exa:3DGOOD_num}
Consider     system~\eqref{eq:3dgood} with~$n=4$, 
 $\alpha_1=\alpha_2=\alpha_3=\alpha
_4=1/2$, and~$m=10$. Then  
\[
\B_G=\{x\in\R^4_{\geq0 } \suchthat x_1\leq  2,\;
x_2\leq  4,\; 
x_3\leq  8,\;x_4\leq  16
\},
\]
and
the polynomial in~\eqref{eq:ps_e3}
becomes 
\[
Q(s)=\frac{1}{16}s^{11}+\frac{1}{16} s-1. 
\]
The unique real and positive root of this polynomial is~$e_4\approx 1.2770$,
and
\[
e=\begin{bmatrix}
   e_4/8&    e_4/4&   e_4/2 & e_4
\end{bmatrix}^\top\approx\begin{bmatrix}
     0.1596&    0.3192 &   0.6385&    1.2770
\end{bmatrix}^\top.
\]
The characteristic polynomial of~$J(e)$ is 
\[
\det(sI_4-J(e)) = s^4+2 s^3 
+1.5 s^2+0.5 s+0.6376  . 
\]
As expected, all the coefficients in this polynomial are positive. Applying   the Routh stability criterion reveals that~$e$ is unstable. Indeed, the eigenvalues of~$J(e)$ are   
\[
  0.1158 + 0.6158j,\
   0.1158 - 0.6158j,\ 
   -1.1158 + 0.6158j,\
  -1.1158 - 0.6158j ,
\]
so~$J(e)$ has two unstable eigenvalues.
Fig.~\ref{fig:uns_ex1} depicts the solution~$x(t,a)$ of~\eqref{eq:3dgood} emanating from the initial condition~$a=\begin{bmatrix}
    0.1 &0.1&0.1&0.1
\end{bmatrix}^\top$ as a function of time.
Note that~$s^-(a-e)=0$. As predicted by our result, the trajectory~$x(t,a)$ converges to a periodic orbit. 
\begin{figure}[t]
\centering
\includegraphics[scale=0.6]{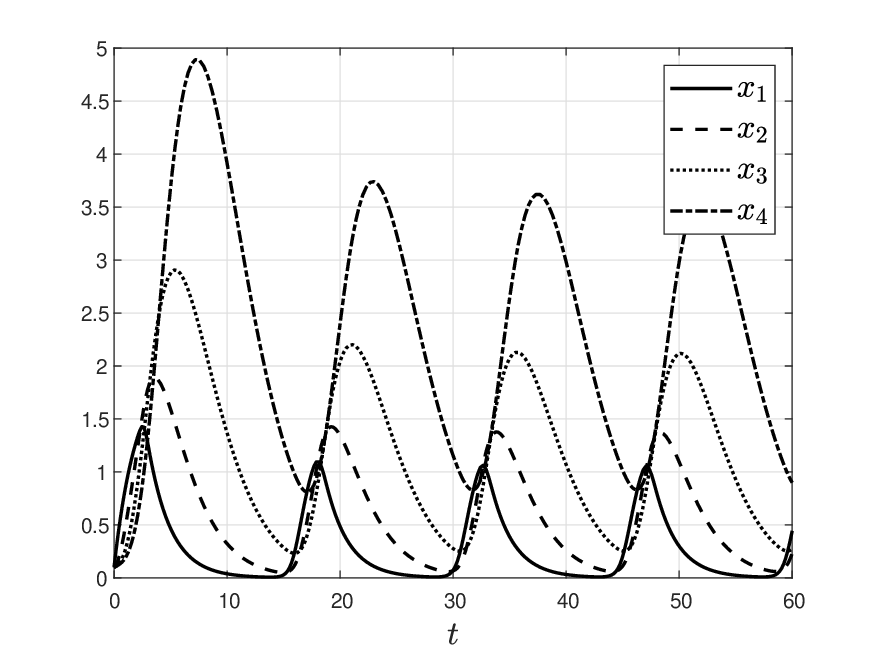}
\caption{Solution to the   Goodwin system in Example~\ref{exa:3DGOOD_num} emanating from~$x(0)=\begin{bmatrix}
    0.1 &0.1&0.1&0.1
\end{bmatrix}^\top$. \label{fig:uns_ex1}}
\end{figure}
\end{Example}

\subsection{Convergence to periodic orbits in a biomolecular oscillator model.} \label{subsec:applic11}
\cite{BCFG2014} proposed the following mathematical model for a biological oscillator based on RNA–mediated regulation: 
\begin{equation}\label{eq:FamilyOfSyst1}
\begin{cases}
\dot{x}_1 & = \kappa_1x_2 -\delta_1x_1-\gamma_2x_4x_1,\\
\dot{x}_2&=-\beta_1x_2 +\gamma_1\left(x_2^{tot}-x_2 \right)x_3,\\
\dot{x}_3 &= \kappa_2x_4-\delta_2 x_3-\gamma_1\left(x_2^{tot}-x_2 \right)x_3 ,\\
\dot{x}_4 &= \beta_2\left(x_4^{tot}-x_4 \right)-\gamma_2 x_4x_1 ,
\end{cases}
\end{equation}
where   $\kappa_i$, $\beta_i$, $\delta_i$ and $\gamma_i$, as well as~$x_2^{tot}$ and $x_4^{tot}$, are   positive parameters, and proved that there exist~$x_1^{tot},x_3^{tot}>0$ such that the closed set 
$\mathcal B_O := [0,x_1^{tot}]\times\dots\times [0,x_4^{tot}]$
is an invariant set of the dynamics, and that~$\mathcal B_O$     includes a unique equilibrium~$e\in\Int(\mathcal B_O)$.   
They also showed that, up to a coordinate  transformation, the system~\eqref{eq:FamilyOfSyst1}  is the negative feedback  interconnection of two cooperative systems, and hence it is structurally a strong candidate oscillator according to the classification introduced by \cite{BFG2014structural,BFG2015structuralclass}, namely, local instability can only occur due to a complex pair
of unstable poles crossing the imaginary axis. For such
an interconnection there exists a \emph{sufficient} condition for the global asymptotic stability of~$e$~\citep{mcs_angeli_2003}. \cite{BCFG2014} suggested that, when this condition does not hold, the system may admit a non-trivial periodic solution, and can thus serve as a biological oscillator. This was motivated by the analysis of the linearization of \eqref{eq:FamilyOfSyst1} around the equilibrium $e$, which demonstrated that, for some choices of $\kappa_i$, $\beta_i$, $\delta_i$ and $\gamma_i$, the eigenvalues of $J(e)$ do cross the imaginary axis in the complex plane. The appearance of oscillations was also confirmed  using extensive  simulations. Applying Theorem~\ref{thm:main_n_dim} provides more precise information of a global nature. 

 \begin{Corollary}
    Consider the $4$-dimensional  oscillator model~\eqref{eq:FamilyOfSyst1}  and let~$e$ denote the unique 
    equilibrium in~$\operatorname{int}\left(\B_O\right)$. Let~$J:\R^n_{\geq 0}\to\R^{n\times n}$ denote the Jacobian of the vector field of system~\eqref{eq:FamilyOfSyst1}.
    Suppose that~$J(e)$ admits at least two eigenvalues with a positive real part. Then for any initial condition~$x_0\in\Int(\mathcal B_O)$ such that~$s^-(x_0-e)\leq 1$  the solution~$x(t,x_0) $ of~\eqref{eq:FamilyOfSyst1} converges to a (non-trivial) periodic orbit as $t\to \infty$.
\end{Corollary}

\begin{pf}
Since the closed set~$\mathcal B_O$ is an invariant set of the dynamics, so is the open set~$\Int(\mathcal B_O)$ (see, e.g.,~\cite{mcs_angeli_2003}).
The Jacobian of the vector field in~\eqref{eq:FamilyOfSyst1}  is 
\begin{equation}\label{eq:FamilyJacob1}
J(x)=\begin{bmatrix}
    -\delta_1-\gamma_2x_4 & \kappa_1 &0&-\gamma_2 x_1\\
    0& -\beta_1-\gamma_1x_3& \gamma_1\left(x_2^{tot}-x_2 \right)&0\\
    0& \gamma_1 x_3 & -\delta_2-\gamma_1(x_2^{tot}-x_2)& \kappa_2\\
    -\gamma_2x_4 & 0& 0 & -\beta_2-\gamma_2x_1
\end{bmatrix}, 
\end{equation}
and hence the system is strongly 2-cooperative on~$\Int(\mathcal B_0)$. Applying Theorem~\ref{thm:main_n_dim} completes the proof. 
\end{pf}

 \begin{Example}\label{exa:bio}
     Consider the system~\eqref{eq:FamilyOfSyst1} with the (arbitrarily chosen) parameter 
     values: 
$\beta_1=0.2$,
$\beta_2=0.5$,
$\kappa_1=15$,
$\kappa_2=1$,
$\delta_1=0.01$,
$\delta_2=0.1$,
$\gamma_1=0.1$,
$\gamma_2=20$,
$x_2^{tot}=15$, and~$x_4^{tot}=20$.
     A numerical calculation shows that in this case~$e=\begin{bmatrix}
  3.4932&
   0.6643&
   0.0927&
   0.1421
     \end{bmatrix}^\top$, and~$J(e) $ 
     admits two unstable eigenvalues. 
     Let~$x_0=0$, and note that~$s^-(e-x_0)=0$. Fig.~\ref{fig:candiate}  depicts~$x(t,x_0)$ as a function of time and shows that it
       converges to a periodic orbit.
 \end{Example}

\begin{figure}[t]
\begin{subfigure}{0.475\linewidth}
  \includegraphics[width=\linewidth]{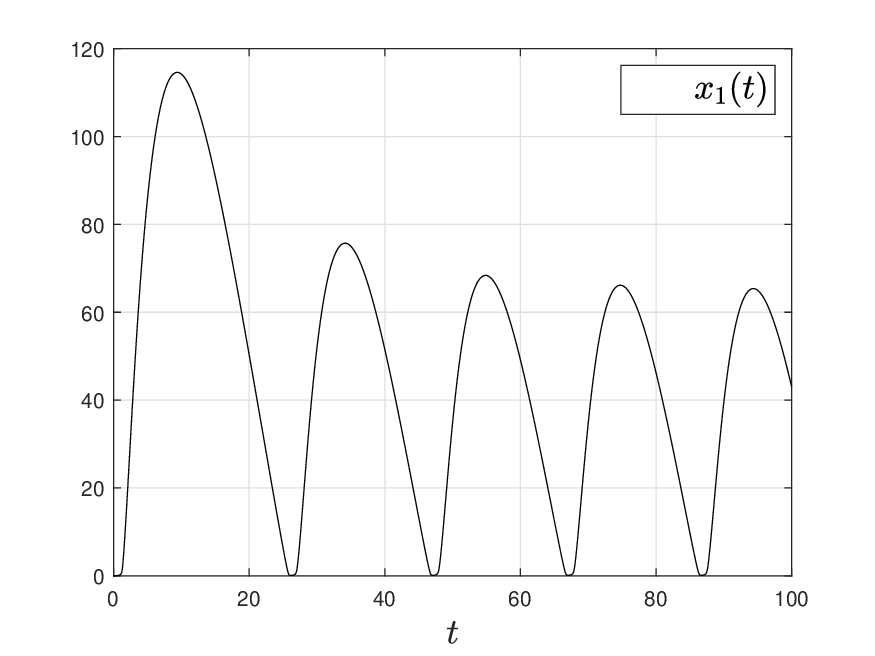}
  \label{MLEDdet}
\end{subfigure}\hfill 
\begin{subfigure}{.475\linewidth}
  \includegraphics[width=\linewidth]{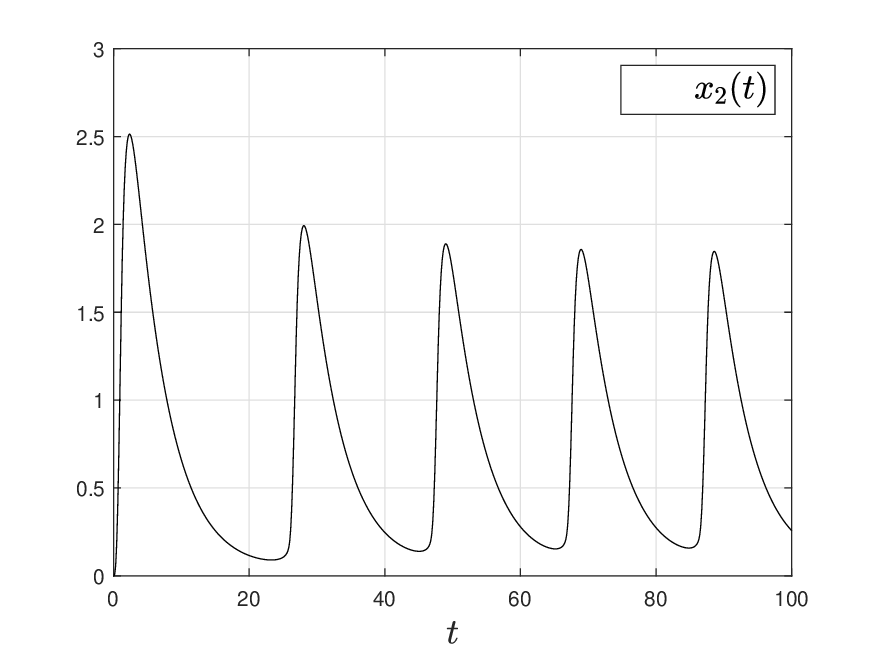}
  \label{energydetPSK}
\end{subfigure}

\medskip 
\begin{subfigure}{.475\linewidth}
  \includegraphics[width=\linewidth]{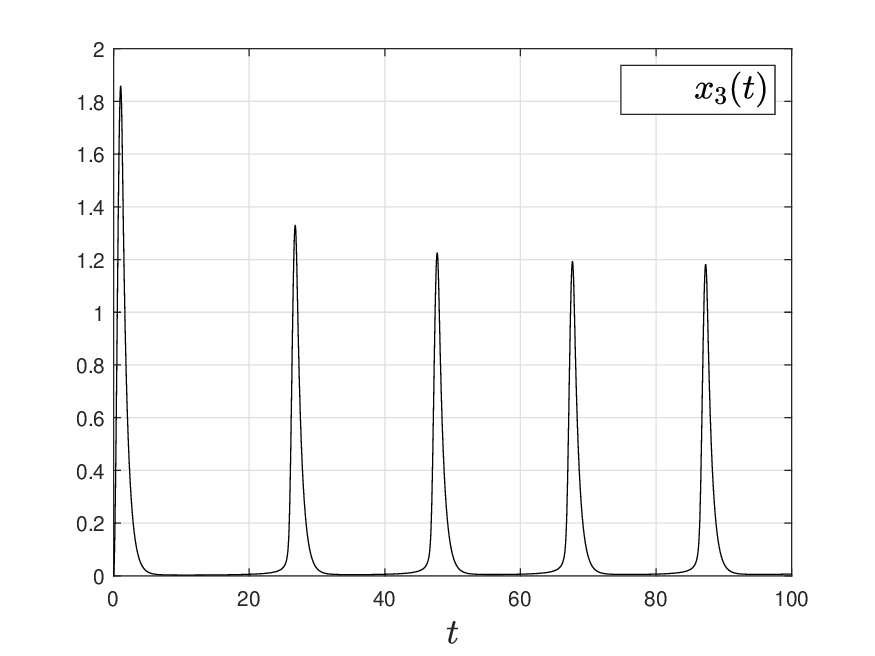}
  \label{velcomp}
\end{subfigure}\hfill 
\begin{subfigure}{.475\linewidth}
  \includegraphics[width=\linewidth]{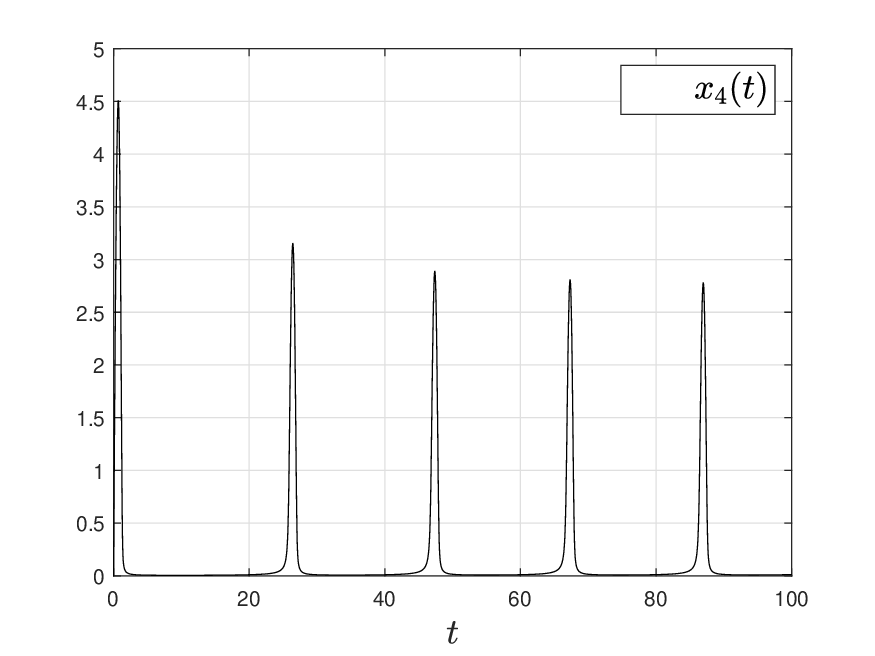}
  \label{estcomp}
\end{subfigure}

\caption{Solution~$x_i(t)$, $i=1,\dots,4$, to the oscillator model in Example~\ref{exa:bio} emanating from $x(0)=0$.}
\label{fig:candiate}
\end{figure}

    \section{Discussion}
 Strongly 2-cooperative systems enjoy  an important asymptotic property, called the strong \Poincare-Bendixon property: any  bounded solution that keeps a positive distance from the set of equilibria converges to a periodic orbit. For a general system satisfying the strong \Poincare-Bendixon property, it is possible that all bounded solutions converge to an equilibrium, so that the system does not  exhibit oscillatory behaviors. Therefore, it is important to find sufficient conditions that guarantee convergence to a non-trivial
periodic orbit.

For a strongly 2-cooperative system defined on a bounded and convex invariant set with a single equilibrium point, we derived a simple condition that guarantees the existence  of periodic trajectories. Moreover, we explicitly characterized a positive measure set of initial conditions such that any bounded solution emanating from this set must converge to a
periodic orbit. The proof relies on the special asymptotic and spectral properties of strongly 2-cooperative systems.
We demonstrated our theoretical  results using two models from systems biology: the $n$-dimensional Goodwin  model and a 4-dimensional biomolecular oscillator.

An   interesting topic for further research is 
to find conditions that guarantee uniqueness of the periodic orbit (see,~ e.g. the survey paper~\cite{LI1981931}). Then, our results would imply that   the periodic orbit is a global attractor for solutions emanating  from an  invariant set that includes no equilibrium points. 
Another research direction may be     applying the theory to 
  the  systemic design of synthetic biological 
  oscillators (see, e.g.~\citep{BCFG2014,novak2008,syn_osci_2020}). 

 \subsection*{Acknowledgments}
 We thank Ron Ofir, Alexander Ovseevich and Florin Avram 
 for helpful comments.


 

\begin{thebibliography}{55}
\providecommand{\natexlab}[1]{#1}
\providecommand{\url}[1]{\texttt{#1}}
\expandafter\ifx\csname urlstyle\endcsname\relax
  \providecommand{\doi}[1]{doi: #1}\else
  \providecommand{\doi}{doi: \begingroup \urlstyle{rm}\Url}\fi

\bibitem[Angeli and Sontag(2003)]{mcs_angeli_2003}
D.~Angeli and E.~D. Sontag.
\newblock Monotone control systems.
\newblock \emph{IEEE Trans.\ Automat.\ Control}, 48:\penalty0 1684--1698, 2003.

\bibitem[Angeli et~al.(2004)Angeli, Ferrell, and Sontag]{Angeli2004}
D.~Angeli, J.~E. Ferrell, and E.~D. Sontag.
\newblock Detection of multistability, bifurcations, and hysteresis in a large class of biological positive-feedback systems.
\newblock \emph{Proceedings of the National Academy of Sciences}, 101\penalty0 (7):\penalty0 1822--1827, 2004.

\bibitem[Bar-Shalom et~al.(2023)Bar-Shalom, Dalin, and Margaliot]{comp_long_tutorial}
E.~Bar-Shalom, O.~Dalin, and M.~Margaliot.
\newblock Compound matrices in systems and control theory: a tutorial.
\newblock \emph{Math. Control Signals Systems}, 35:\penalty0 467--521, 2023.

\bibitem[Ben~Avraham et~al.(2020)Ben~Avraham, Sharon, Zarai, and Margaliot]{cyclic_sign_vari}
T.~Ben~Avraham, G.~Sharon, Y.~Zarai, and M.~Margaliot.
\newblock Dynamical systems with a cyclic sign variation diminishing property.
\newblock \emph{IEEE Trans.\ Automat.\ Control}, 65\penalty0 (3):\penalty0 941--954, 2020.

\bibitem[Blanchini and Giordano(2021)]{BLANCHINI_GIORDANO_SURVEY}
F.~Blanchini and G.~Giordano.
\newblock Structural analysis in biology: A control-theoretic approach.
\newblock \emph{Automatica}, 126:\penalty0 109376, 2021.

\bibitem[Blanchini et~al.(2014{\natexlab{a}})Blanchini, Cuba~Samaniego, Franco, and Giordano]{BCFG2014}
F.~Blanchini, C.~Cuba~Samaniego, E.~Franco, and G.~Giordano.
\newblock Design of a molecular clock with {RNA}-mediated regulation.
\newblock In \emph{Proc. 53rd IEEE Conference on Decision and Control}, pages 4611--4616, 2014{\natexlab{a}}.

\bibitem[Blanchini et~al.(2014{\natexlab{b}})Blanchini, Franco, and Giordano]{BFG2014structural}
F.~Blanchini, E.~Franco, and G.~Giordano.
\newblock A structural classification of candidate oscillatory and multistationary biochemical systems.
\newblock \emph{Bulletin of Mathematical Biology}, 76:\penalty0 2542--2569, 2014{\natexlab{b}}.

\bibitem[Blanchini et~al.(2015)Blanchini, Franco, and Giordano]{BFG2015structuralclass}
F.~Blanchini, E.~Franco, and G.~Giordano.
\newblock Structural conditions for oscillations and multistationarity in aggregate monotone systems.
\newblock In \emph{Proceedings of the IEEE Conference on Decision and Control}, 2015.

\bibitem[Blekhman(2000)]{Blekhman2000}
I.~I. Blekhman.
\newblock \emph{Vibrational Mechanics}.
\newblock World Scientific, 2000.

\bibitem[Donnell et~al.(2009)Donnell, Baigent, and Banaji]{Donnell2009120}
P.~Donnell, S.~A. Baigent, and M.~Banaji.
\newblock Monotone dynamics of two cells dynamically coupled by a voltage-dependent gap junction.
\newblock \emph{J. Theoretical Biology}, 261\penalty0 (1):\penalty0 120--125, 2009.

\bibitem[Doyle et~al.(2006)Doyle, Gunawan, Bagheri, Mirsky, and To]{Doyle2006}
F.~J. Doyle, R.~Gunawan, N.~Bagheri, H.~Mirsky, and T.~L. To.
\newblock Circadian rhythm: A natural, robust, multi-scale control system.
\newblock \emph{Comput Chem Eng}, 30:\penalty0 1700--1711, 2006.

\bibitem[Fallat and Johnson(2011)]{total_book}
S.~M. Fallat and C.~R. Johnson.
\newblock \emph{Totally Nonnegative Matrices}.
\newblock Princeton University Press, Princeton, NJ, 2011.

\bibitem[Farkas(1994)]{frakas_periodic_book}
M.~Farkas.
\newblock \emph{Periodic Motions}, volume 104 of \emph{Applied Mathematical Sciences}.
\newblock Springer, 1994.

\bibitem[Feng et~al.(2021)Feng, Wang, and Wu]{FENG2021858}
L.~Feng, Y.~Wang, and J.~Wu.
\newblock Generic behavior of flows strongly monotone with respect to high-rank cones.
\newblock \emph{J. Diff. Eqns.}, 275:\penalty0 858--881, 2021.

\bibitem[Ferrell et~al.(2011)Ferrell, Tsai, and Yang]{Ferrell2011}
J.~Ferrell, T.-C. Tsai, and Q.~Yang.
\newblock Modeling the cell cycle: Why do certain circuits oscillate?
\newblock \emph{Cell}, 144:\penalty0 874 – 885, 2011.

\bibitem[Fusco and Oliva(1991)]{fusco_oliva_1991}
G.~Fusco and W.~M. Oliva.
\newblock A {Perron} theorem for the existence of invariant subspaces.
\newblock \emph{Annali di Matematica Pura ed Applicata}, 160:\penalty0 63--76, 1991.

\bibitem[Gantmacher and Krein(2002)]{gk_book}
F.~R. Gantmacher and M.~G. Krein.
\newblock \emph{Oscillation Matrices and Kernels and Small Vibrations of Mechanical Systems}.
\newblock American Mathematical Society, Providence, RI, 2002.
\newblock Translation based on the~1941 {Russian} original.

\bibitem[Ge and Arcak(2009)]{ARCAK_D_STAB}
X.~Ge and M.~Arcak.
\newblock A sufficient condition for additive {D}-stability and application to reaction–diffusion models.
\newblock \emph{Systems Control Lett.}, 58\penalty0 (10):\penalty0 736--741, 2009.

\bibitem[Goldbeter et~al.(2012)Goldbeter, Gérard, Gonze, Leloup, and Dupont]{Goldbeter2012}
A.~Goldbeter, C.~Gérard, D.~Gonze, J.~C. Leloup, and G.~Dupont.
\newblock Systems biology of cellular rhythms.
\newblock \emph{FEBS Lett}, 586:\penalty0 2955--2965, 2012.

\bibitem[Gonze and Ruoff(2021)]{gonze2021}
D.~Gonze and P.~Ruoff.
\newblock The {Goodwin} oscillator and its legacy.
\newblock \emph{Acta Biotheor.}, 6\penalty0 (4):\penalty0 857--874, 2021.

\bibitem[Goodwin(1965)]{GOODWIN1965425}
B.~C. Goodwin.
\newblock Oscillatory behavior in enzymatic control processes.
\newblock \emph{Advances in Enzyme Regulation}, 3:\penalty0 425--438, 1965.

\bibitem[Hastings et~al.(1977)Hastings, Tyson, and Webster]{HASTINGS1977}
S.~Hastings, J.~Tyson, and D.~Webster.
\newblock Existence of periodic solutions for negative feedback cellular control systems.
\newblock \emph{J. Diff. Eqns.}, 25\penalty0 (1):\penalty0 39--64, 1977.

\bibitem[Horn and Johnson(2013)]{matrx_ana}
R.~A. Horn and C.~R. Johnson.
\newblock \emph{Matrix Analysis}.
\newblock Cambridge University Press, 2 edition, 2013.

\bibitem[Invernizzi and Treu(1991)]{hopf_Goodwin_n_dim}
S.~Invernizzi and G.~Treu.
\newblock Quantitative analysis of the {Hopf} bifurcation in the {Goodwin} n-dimensional metabolic control system.
\newblock \emph{J. Math. Biol.}, 29:\penalty0 733–742, 1991.

\bibitem[Johnson-Buck and Shih(2017)]{Johnson2017}
A.~Johnson-Buck and W.~M. Shih.
\newblock Single-molecule clocks controlled by serial chemical reactions.
\newblock \emph{Nano Letters}, 17:\penalty0 7940--7944, 2017.

\bibitem[Katz et~al.(2024)Katz, Giordano, and Margaliot]{rami_3dGoddwin_2024}
R.~Katz, G.~Giordano, and M.~Margaliot.
\newblock Existence of attracting periodic orbits in 3-dimensional strongly 2-cooperative~systems.
\newblock 2024.
\newblock URL \url{https://arxiv.org/abs/2407.00461}.
\newblock Submitted.

\bibitem[Krasnosel\'skii et~al.(1989)Krasnosel\'skii, Lifshits, and Sobolev]{KLS89}
M.~Krasnosel\'skii, E.~Lifshits, and A.~Sobolev.
\newblock \emph{Positive Linear Systems, the Method of Positive Operators}.
\newblock Heldermann Verlag, Berlin, 1989.

\bibitem[Leenheer et~al.(2007)Leenheer, Angeli, and Sontag]{mono_chem_2007}
P.~D. Leenheer, D.~Angeli, and E.~D. Sontag.
\newblock Monotone chemical reaction networks.
\newblock \emph{J. Mathematical Chemistry}, 41:\penalty0 295--314, 2007.

\bibitem[Li(1981)]{LI1981931}
B.~Li.
\newblock Periodic orbits of autonomous ordinary differential equations: theory and applications.
\newblock \emph{Nonlinear Analysis: Theory, Methods \& Applications}, 5\penalty0 (9):\penalty0 931--958, 1981.

\bibitem[Mallet-Paret and Smith(1990)]{poin_cyclic}
J.~Mallet-Paret and H.~L. Smith.
\newblock The {Poincar\'e-Bendixson} theorem for monotone cyclic feedback systems.
\newblock \emph{J. Dyn. Differ. Equ.}, 2\penalty0 (4):\penalty0 367--421, 1990.

\bibitem[Margaliot and Sontag(2019)]{fulltppaper}
M.~Margaliot and E.~D. Sontag.
\newblock Revisiting totally positive differential systems: A tutorial and new results.
\newblock \emph{Automatica}, 101:\penalty0 1--14, 2019.

\bibitem[Margaliot and Tuller(2012)]{RFM_stability}
M.~Margaliot and T.~Tuller.
\newblock Stability analysis of the ribosome flow model.
\newblock \emph{IEEE/ACM Trans. Comput. Biol. Bioinf.}, 9:\penalty0 1545--1552, 2012.

\bibitem[Novák and Tyson(2008)]{novak2008}
B.~Novák and J.~J. Tyson.
\newblock Design principles of biochemical oscillators.
\newblock \emph{Nat Rev Mol Cell Biol}, 9\penalty0 (12):\penalty0 981--91, 2008.

\bibitem[Panghalia and Singh(2020)]{syn_osci_2020}
A.~Panghalia and V.~Singh.
\newblock Design principles of synthetic biological oscillators.
\newblock In V.~Singh, editor, \emph{Advances in Synthetic Biology}. Springer, Singapore, 2020.

\bibitem[Pinkus(2010)]{pinkus}
A.~Pinkus.
\newblock \emph{Totally Positive Matrices}.
\newblock Cambridge University Press, Cambridge, UK, 2010.

\bibitem[Raveh et~al.(2016)Raveh, Margaliot, Sontag, and Tuller]{RFM_model_compete_J}
A.~Raveh, M.~Margaliot, E.~D. Sontag, and T.~Tuller.
\newblock A model for competition for ribosomes in the cell.
\newblock \emph{J. Royal Society Interface}, 13\penalty0 (116), 2016.

\bibitem[Sanchez(2009{\natexlab{a}})]{Sanchez2009ConesOR}
L.~Sanchez.
\newblock Cones of rank 2 and the {P}oincar{\'e}–{B}endixson property for a new class of monotone systems.
\newblock \emph{J. Diff. Eqns.}, 246:\penalty0 1978--1990, 2009{\natexlab{a}}.

\bibitem[Sanchez(2009{\natexlab{b}})]{sanchez_goodein_stab}
L.~A. Sanchez.
\newblock Global asymptotic stability of the {Goodwin} system with repression.
\newblock \emph{Nonlinear analysis: real world applications}, 10\penalty0 (4):\penalty0 2151--2156, 2009{\natexlab{b}}.

\bibitem[Schubert and Kim(2016)]{Schubert2016}
T.~F. Schubert and E.~M. Kim.
\newblock \emph{Fundamentals of Electronics: Oscillators and Advanced Electronics Topics (Book 4)}.
\newblock Springer Cham, 2016.

\bibitem[Shi et~al.(2019)Shi, Altafini, and Baras]{altafini_survey}
G.~Shi, C.~Altafini, and J.~S. Baras.
\newblock Dynamics over signed networks.
\newblock \emph{SIAM Review}, 61\penalty0 (2):\penalty0 229--257, 2019.

\bibitem[Smith(1991)]{smith_neural}
H.~L. Smith.
\newblock Convergent and oscillatory activation dynamics for cascades of neural nets with nearest neighbor competitive or cooperative interactions.
\newblock \emph{Neural Networks}, 4:\penalty0 41--46, 1991.

\bibitem[Smith(1995)]{hlsmith}
H.~L. Smith.
\newblock \emph{Monotone Dynamical Systems: An Introduction to the Theory of Competitive and Cooperative Systems}, volume~41 of \emph{Mathematical Surveys and Monographs}.
\newblock Amer. Math. Soc., Providence, RI, 1995.

\bibitem[Smith(1987)]{Smith_orbital_stability}
R.~A. Smith.
\newblock Orbital stability for ordinary differential equations.
\newblock \emph{J. Diff. Eqns.}, 69\penalty0 (2):\penalty0 265--287, 1987.

\bibitem[Sontag(1998)]{sontag_book}
E.~D. Sontag.
\newblock \emph{Mathematical Control Theory: Deterministic Finite Dimensional Systems}.
\newblock Springer, New York, 2 edition, 1998.

\bibitem[Sontag(2007)]{sontag_near_2007}
E.~D. Sontag.
\newblock Monotone and near-monotone biochemical networks.
\newblock \emph{Systems and Synthetic Biology}, 1:\penalty0 59--87, 2007.

\bibitem[Teschl(2012)]{diff_eqn_limit_sets}
G.~Teschl.
\newblock \emph{Ordinary Differential Equations and Dynamical Systems}, volume 140 of \emph{Graduate Studies in Mathematics}.
\newblock American Mathematical Society, 2012.

\bibitem[Tyson(1975)]{Tyson_3D}
J.~J. Tyson.
\newblock On the existence of oscillatory solutions in negative feedback cellular control processes.
\newblock \emph{J. Math. Biology}, 1:\penalty0 311--315, 1975.

\bibitem[Uriu et~al.(2009)Uriu, Morishita, and Iwasa]{Uriu2009}
K.~Uriu, Y.~Morishita, and Y.~Iwasa.
\newblock Traveling wave formation in vertebrate segmentation.
\newblock \emph{Journal of Theoretical Biology}, 257\penalty0 (3):\penalty0 385--396, 2009.

\bibitem[Uriu et~al.(2010)Uriu, Morishita, and Iwasa]{Uriu2010}
K.~Uriu, Y.~Morishita, and Y.~Iwasa.
\newblock Synchronized oscillation of the segmentation clock gene in vertebrate development.
\newblock \emph{Journal of Mathematical Biology}, 61:\penalty0 207–229, 2010.

\bibitem[Wang et~al.(2022)Wang, Yao, and Zhang]{wang_cyclic_feedback}
Y.~Wang, J.~Yao, and Y.~Zhang.
\newblock Prevalent behavior and almost sure {Poincar\'e–Bendixson} theorem for smooth flows with invariant $k$-cones.
\newblock \emph{J. Dyn. Diff. Equat.}, 2022.

\bibitem[Weiss and Margaliot(2021)]{WEISS_k_posi}
E.~Weiss and M.~Margaliot.
\newblock A generalization of linear positive systems with applications to nonlinear systems: Invariant sets and the {Poincaré–Bendixson} property.
\newblock \emph{Automatica}, 123:\penalty0 109358, 2021.

\bibitem[Woller et~al.(2014)Woller, Gonze, and Erneux]{Woller_2014_3D_Goodwin}
A.~Woller, D.~Gonze, and T.~Erneux.
\newblock The {Goodwin} model revisited: {Hopf} bifurcation, limit-cycle, and periodic entrainment.
\newblock \emph{Physical Biology}, 11\penalty0 (4):\penalty0 045002, 2014.

\bibitem[Wooldridge(2009)]{wooldridge2009introduction}
M.~Wooldridge.
\newblock \emph{An introduction to Multiagent Systems}.
\newblock John Wiley \& Sons, 2009.

\bibitem[Wu et~al.(2022{\natexlab{a}})Wu, Kanevskiy, and Margaliot]{kordercont}
C.~Wu, I.~Kanevskiy, and M.~Margaliot.
\newblock $k$-contraction: theory and applications.
\newblock \emph{Automatica}, 136:\penalty0 110048, 2022{\natexlab{a}}.

\bibitem[Wu et~al.(2022{\natexlab{b}})Wu, Pines, Margaliot, and Slotine]{Hausdorff_contract}
C.~Wu, R.~Pines, M.~Margaliot, and J.-J. Slotine.
\newblock Generalization of the multiplicative and additive compounds of square matrices and contraction theory in the {Hausdorff} dimension.
\newblock \emph{IEEE Trans.\ Automat.\ Control}, 67\penalty0 (9):\penalty0 4629--4644, 2022{\natexlab{b}}.

\end{thebibliography}
 \end{document}